\documentclass[onefignum,onetabnum]{siamart220329}
\usepackage[margin=1in]{geometry}

\usepackage{braket,amsfonts}

\usepackage{array}

\usepackage[caption=false]{subfig}
\usepackage{pgfplots}

\newsiamthm{claim}{Claim}
\newsiamremark{remark}{Remark}
\newsiamremark{hypothesis}{Hypothesis}
\crefname{hypothesis}{Hypothesis}{Hypotheses}

\usepackage{algpseudocode}

\usepackage{graphicx,epstopdf}

\Crefname{ALC@unique}{Line}{Lines}

\usepackage{amsopn}

\usepackage{xspace}
\usepackage{bold-extra}
\usepackage[most]{tcolorbox}

\colorlet{texcscolor}{blue!50!black}
\colorlet{texemcolor}{red!70!black}
\colorlet{texpreamble}{red!70!black}
\colorlet{codebackground}{black!25!white!25}

\DeclareMathOperator\Arg{Arg}

\newmuskip\pFqmuskip

\newcommand*\pFq[6][8]{%
  \begingroup 
  \pFqmuskip=#1mu\relax
  \mathchardef\normalcomma=\mathcode`,
  \mathcode`\,=\string"8000
  \begingroup\lccode`\~=`\,
  \lowercase{\endgroup\let~}\pFqcomma
  {}_{#2}F_{#3}{\left[\genfrac..{0pt}{}{#4}{#5};#6\right]}%
  \endgroup
}
\newcommand{\pFqcomma}{{\normalcomma}\mskip\pFqmuskip}

\usepackage{color}

\lstdefinestyle{siamlatex}{%
  style=tcblatex,
  texcsstyle=*\color{texcscolor},
  texcsstyle=[2]\color{texemcolor},
  keywordstyle=[2]\color{texemcolor},
  moretexcs={cref,Cref,maketitle,mathcal,text,headers,email,url},
}

\tcbset{%
  colframe=black!75!white!75,
  coltitle=white,
  colback=codebackground, 
  colbacklower=white, 
  fonttitle=\bfseries,
  arc=0pt,outer arc=0pt,
  top=1pt,bottom=1pt,left=1mm,right=1mm,middle=1mm,boxsep=1mm,
  leftrule=0.3mm,rightrule=0.3mm,toprule=0.3mm,bottomrule=0.3mm,
  listing options={style=siamlatex}
}

\newtcblisting[use counter=example]{example}[2][]{%
  title={Example~\thetcbcounter: #2},#1}

\newtcbinputlisting[use counter=example]{\examplefile}[3][]{%
  title={Example~\thetcbcounter: #2},listing file={#3},#1}

\DeclareTotalTCBox{\code}{ v O{} }
{ 
  fontupper=\ttfamily\color{black},
  nobeforeafter,
  tcbox raise base,
  colback=codebackground,colframe=white,
  top=0pt,bottom=0pt,left=0mm,right=0mm,
  leftrule=0pt,rightrule=0pt,toprule=0mm,bottomrule=0mm,
  boxsep=0.5mm,
  #2}{#1}

\patchcmd\newpage{\vfil}{}{}{}
\begin{tcbverbatimwrite}{tmp_\jobname_header.tex}
\title{Fast and robust method for screened Poisson lattice Green's function using asymptotic expansion and Fast Fourier Transform\thanks{\textbf{Funding}~{This work was supported in part by The Boeing Company (CT-BA-GTA-1).} }}

\author{Wei Hou\thanks{Division of Engineering and Applied Science, California Institute of Technology (\email{whou@caltech.edu}).}
\and Tim Colonius\thanks{Division of Engineering and Applied Science, California Institute of Technology  (\email{colonius@caltech.edu}).}}

\headers{Screened Poisson LGF using Asymptotic Expansion and FFT}{Wei Hou, Tim Colonius}
\end{tcbverbatimwrite}
\input{tmp_\jobname_header.tex}

\ifpdf
\hypersetup{ pdftitle={Fast and robust method for screened Poisson LGF} }
\fi

\begin{document}
\maketitle

\begin{tcbverbatimwrite}{tmp_\jobname_abstract.tex}
\begin{abstract}
We study the lattice Green's function (LGF) of the screened Poisson equation on a two-dimensional rectangular lattice. This LGF arises in numerical analysis, random walks, solid-state physics, and other fields. Its defining characteristic is the screening term, which defines different regimes. When its coefficient is large, we can accurately approximate the LGF with an exponentially converging asymptotic expansion, and its convergence rate monotonically increases with the coefficient of the screening term. To tabulate the LGF when the coefficient is not large, we derive a one-dimensional integral representation of the LGF. We show that the trapezoidal rule can approximate this integral with exponential convergence, and we propose an efficient algorithm for its evaluation via the Fast Fourier Transform.  We discuss applications including computing the LGF of the three-dimensional Poisson equation with one periodic direction and the return probability of a two-dimensional random walk with killing.
\end{abstract}


\end{tcbverbatimwrite}
\input{tmp_\jobname_abstract.tex}

\section{Introduction} 
The discrete screened Poisson equation for a $k$-dimensional space with parameter $c^2 > 0$ is defined as
\begin{equation}
    L_c u(\boldsymbol{n}) = c^2 u(\boldsymbol{n}) + \sum\limits_{j=1}^k \left[-\alpha_j u(\boldsymbol{n}-\boldsymbol{e}_j) + 2\alpha_j u(\boldsymbol{n}) - \alpha_ju(\boldsymbol{n}+\boldsymbol{e}_j) \right] = f(\boldsymbol{n}) \quad \forall \boldsymbol{n} \in \mathbb{Z}^k,
    \label{eq:LC_OPERATOR_gen}
\end{equation}
where $\boldsymbol{e}_1$, ..., $\boldsymbol{e}_k$ are the coordinate vectors of $\mathbb{R}^k$, and $\alpha_1$,...,$\alpha_k > 0$ are the anisotropy coefficients. The $c^2$ term is sometimes called the screening term \cite{kazhdan2013screened}. The LGF is the fundamental solution of the equation above. It plays a role in physics \cite{kotera1962localized, katsura1971lattice, katsura1973asymptotic}, mathematics \cite{madras1989random, lawler2010random}, and engineering \cite{cserti2000application, kazhdan2013screened}. Theoretical aspects of the LGF of the screened Poisson equation have been studied extensively \cite{morita1971calculation, katsura1971lattice, katsura1973asymptotic, michta2021asymptotic, maassarani2000series, duffin1953discrete,duffin1958difference}. While the LGF of the Poisson equation has an asymptotic expansion at arbitrary orders \cite{duffin1953discrete, duffin1958difference,martinsson2002asymptotic}, the LGF of the screened Poisson equation does not \cite{gabbard2024lattice}. 

In this work, we focus on computational aspects of the LGF of the screened Poisson equation on rectangular lattices. After reviewing previous results (\cref{sec:lgf_def}), we find an asymptotic expansion in terms of the associated value of $c^2$ and establish the decay rate when $c^2$ is relatively large (\cref{sec:Large_c}).  Next, for small $c^2$, we derive a one-dimensional integral representation of the LGF (\cref{sec:LGFcomp}). The same one-dimensional integral representation is also applicable to the LGF of the Poisson equation. We then show that, for screened Poisson equation LGF, the error of a trapezoidal rule approximation can be strictly bounded and converges exponentially fast (\cref{sec:TrapConv}). By exploiting the structure of the integrand, we propose a Fast Fourier Transform \cite{cooley1965algorithm} method for batch evaluation of the LGF (\cref{sec:num_results}). We show that our algorithm is robust and highly efficient.  

We provide two examples to demonstrate how our algorithm can be used in practice. The first example (\cref{sec:app1}) is to use the LGF of the screened Poisson equation to solve for the LGF of the three-dimensional Poisson equation with one periodic dimension. The second example (\cref{sec:app2}) is to use the LGF of the screened Poisson equation to compute the return probability of a two-dimensional random walk with killing.

\section{Lattice Green's function of the two-dimensional screened Poisson equation}
\label{sec:lgf_def}
\subsection{Definition and two-dimensional integral representation}

We consider \cref{eq:LC_OPERATOR_gen} with $k=2$, denote $\boldsymbol{n} = (n,m)$, and use $u(\boldsymbol{n})$ and $u(n,m)$ interchangeably.  Since $\alpha_1, \alpha_2 > 0$, by re-scaling the coefficients, we can make $\alpha_1 = 1$ or $\alpha_2 = 1$. Thus, without loss of generality, we assume $0 < \alpha_1 \leq 1$ and $\alpha_2 = 1$.

The LGF, denoted as $B_c(\boldsymbol{n}) = B_c(n,m)$, is the solution of
\begin{equation}
    [L_c B_c](n,m) = \delta_{0n}\delta_{0m}, \quad \lim\limits_{|n|+|m| \rightarrow \infty} B_c(n,m) = 0,
    \label{eq:def_LGF}
\end{equation}
where $\delta_{ij}$ denotes the Kronecker delta. The Fourier transform method allows the solution to be written as \cite{katsura1971lattice}
\begin{equation}
     B_c(\boldsymbol{n}) = \frac{1}{(2\pi)^2}\int_{I^2} e^{-i\boldsymbol{n}\cdot \boldsymbol{\xi}} \frac{1}{\sigma (\boldsymbol{\xi}) + c^2}d\boldsymbol{\xi},
     \label{eq:int_2d_form}
\end{equation}
where $I^2 = [-\pi,\pi]^2$ is the integration domain, and with $\boldsymbol{\xi} = (\xi_1, \xi_2)$, the function $\sigma(\boldsymbol{\xi})$ is:
\begin{equation}
    \sigma(\boldsymbol{\xi}) = 2\alpha_1 - 2\alpha_1\cos(\xi_1) + 2 - 2\cos(\xi_2).
\end{equation}
From \cref{eq:int_2d_form}, it is clear that $B_c(n,m) = B_c(|n|,|m|)$ so it suffices to consider $n,m \geq 0$. 

For $c = 0$, we recover the standard 2D Poisson equation, and the corresponding LGF can be represented using an asymptotic expansion valid to arbitrarily high order-of-accuracy \cite{martinsson2002asymptotic}. In practice, one typically precomputes LGF near-field values and uses these asymptotic expansions for far-field values \cite{liska2014parallel}.  However, when $c \neq 0$, a high-order asymptotic expansion does not exist, and we must rely on numerical integration for all values.  Thus, efficient computation is essential.

\subsection{Representations using special functions}
The values of $B_c$ can be expressed using Appell's double hypergeometric functions ($F_4$):
\begin{theorem}
    (Katsura and Inawashiro \cite{katsura1971lattice}).The solution of $B_c(n,m)$ can be written as
    \begin{multline}
        B_c(n,m) = \frac{1}{2^{m+n+1}}\frac{1}{a^{m+n+1}}\frac{(m+n)!}{n!m!}\frac{1}{a}\left(\frac{\alpha_1}{a}\right)^n \left(\frac{1}{a}\right)^m \\
        F_4\left[(m+n+1)/2,  (m+n)/2+1, n+1, m+1; \left(\frac{\alpha_1}{a}\right)^2, \left(\frac{1}{a}\right)^2\right],
    \label{eq:appell_LGF}
    \end{multline}
    where
    \begin{equation}
        a = 1+\alpha_1+c^2/2.
    \end{equation}
    \label{thm:appell_LGF}
\end{theorem}
In the special case of $\alpha_1 = 1$, we have a more concise result: 
\begin{theorem}
    (Katsura and Inawashiro \cite{katsura1971lattice}) When $\alpha_1 = 1$, the solution of $B_c(n,m)$ can be written in terms of the generalized hypergeometric function $\,_4F_3$
    \begin{multline}
        B_c(n,m) = \frac{1}{2^{m+n+1}}\frac{1}{a^{m+n+1}}\frac{(m+n)!}{n!m!} \\
        \pFq{4}{3}{(m+n+1)/2,(m+n+1)/2,(m+n)/2,(m+n)/2}{m+1. n+1, m+n+1}{\left(\frac{2}{a}\right)^2}.
    \label{eq:hyper_geom_LGF}
    \end{multline}
\end{theorem}
The first-order asymptotic expansion of $B_c$ is known:
\begin{theorem}
    (Katsura and Inawashiro \cite{katsura1973asymptotic}) The asymptotic form of $B_c(n,m)$, when $\alpha_1 = 1$, at large $|n| + |m|$ is
    \begin{multline}
        B_c(n,m) \sim 2^{-1}(2\pi r)^{-1/2}x^{-1/4}[\mu^2(1+\nu^2x)^{1/2} + \nu^2(1+\mu^2x)^{1/2}]^{-1/2} \\
        \times \exp{\{-r[\mu \cosh^{-1}(\sqrt{1+\mu^2x}) + \nu \cosh^{-1}(\sqrt{1+\nu^2x})]\}},
        \label{eq:asym_exp_Katsura}
    \end{multline}
    where
    \begin{equation}
        r = \sqrt{m^2+n^2},\, \mu = m/r, \, \nu = n/r,
    \end{equation}
    and with $a = 2+c^2/2$
    \begin{equation}
        x = \frac{a^2 - 4}{1+[1 - (1 - 4a^{-2})(\mu^2 - \nu^2)^2]^{1/2}}.
    \end{equation}
    \label{thm:asym_exp_Katsura}
\end{theorem}

The asymptotic form of $B_c$ derived in \cref{thm:asym_exp_Katsura} uses the first term of the asymptotic expansion of the modified Bessel function. Thus, the relative error of this asymptotic form decays at the rate of $1/\sqrt{m^2+n^2}$. Consequently, only when $|m|+|n|$ is large can this asymptotic form accurately approximate the LGF. These statements are demonstrated in \cref{fig:Err_Katsura_Inawashiro}.
\begin{figure}
  \centering
  \subfloat[Relative error at $c = 0.001$.]{\includegraphics[width=0.5\textwidth]{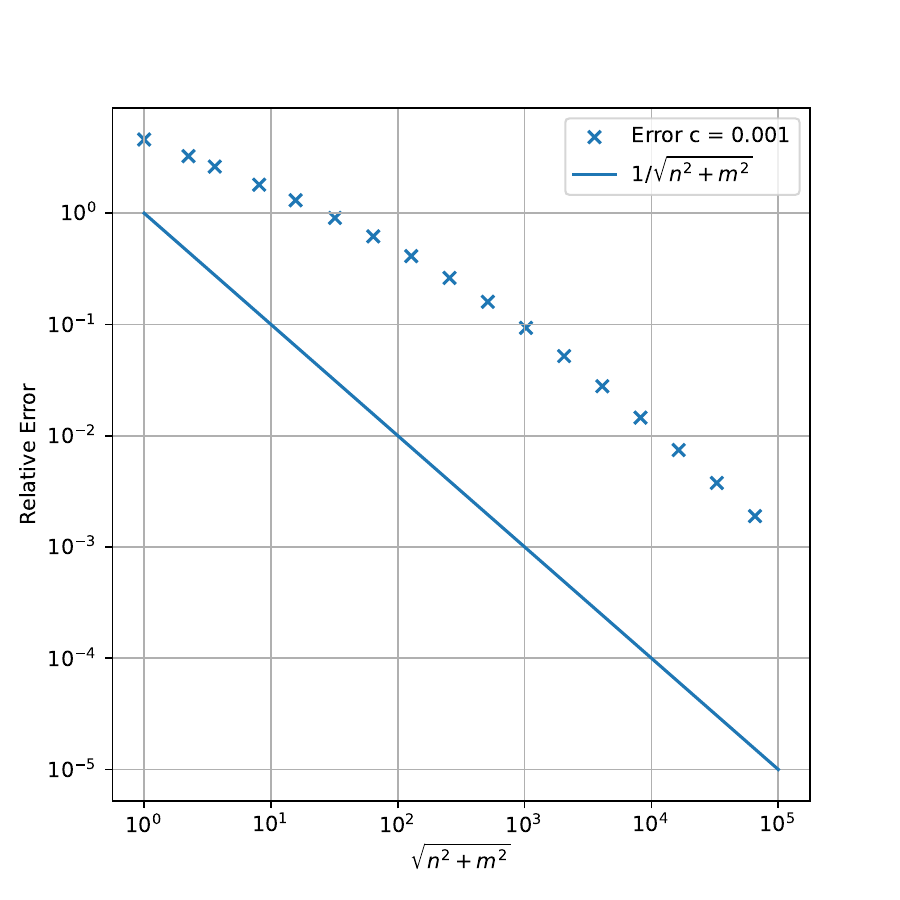}}
  \subfloat[Absolute error at small $|m| + |n|$.]  {\includegraphics[width=0.5\textwidth]{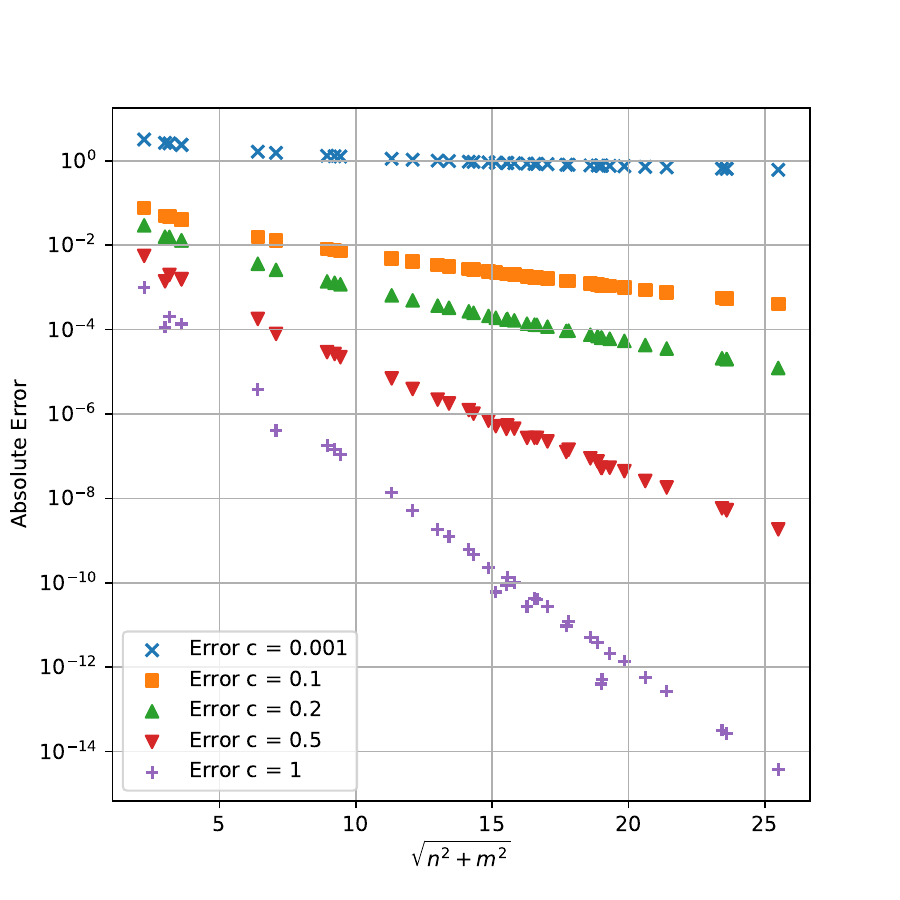}}
  \caption{Asymptotic and near field error of the asymptotic form given in \cref{thm:asym_exp_Katsura}.}
  \label{fig:Err_Katsura_Inawashiro}
\end{figure}
In practice, this leaves many values of $m$ and $n$ for which $B_c(n,m)$ must be integrated directly.  In addition, evaluating $B_c$ when $\alpha_1 < 1$ requires evaluating Appell's double hypergeometric function which is not available in common libraries and contains a doubly infinite sum.  Thus in most numerical applications of the LGF, the Bessel function representation \cite{koster1954simplified,maradudin1960green, katsura1971lattice, maassarani2000series, delves2001green} of $B_c(n,m)$ is used to compute the LGF \cite{liska2017fast, dorschner2020fast, yu2022multi, gabbard2024lattice}. The Bessel function representation is a way to write $B_c(n,m)$ as an improper integral of the Bessel function. The specific formulation reads
    \begin{equation}
        B_c(n,m) = i^{n+m+1} \frac{1}{2}\int_0^\infty e^{-i(2+2\alpha_1+c^2)t/2} J_n(\alpha_1 t)J_m(t) dt,
        \label{eq:bessel_representation}
    \end{equation}
    where $J_k(t)$ is the Bessel function of the first kind \cite{abramowitz1948handbook}.
    It can be further simplified to a more computationally advantageous form:
    \begin{equation}
        B_c(n,m) = \int_0^\infty e^{-(2+2\alpha_1+c^2) t} I_n(2\alpha_1 t)I_m(2t) dt,
        \label{eq:bessel_representation2}
    \end{equation}
where $I_k(t)$ is the modified Bessel function of the first kind \cite{abramowitz1948handbook}.

When evaluating the above integral, there are two challenges: effectively evaluating the function $I_n$ and accurately approximating the improper integral. In practice, for the first challenge, one can compute the modified Bessel function using existing numerical libraries \cite{schaling2011boost, 2020SciPy-NMeth, mpmath}. However, this function is still defined as an infinite series or integral and can be computationally expensive. For the second challenge, the improper integral can either be approximated by integrating up to a large value or be evaluated using a variable transformation \cite{de1978adaptive}. The former can be computationally expensive and unstable, and the latter can create a singularity at the origin. Indeed, it has been reported that numerical evaluation of the LGF of the screened Poisson equation can fail catastrophically \cite{gabbard2024lattice}. In the rest of the paper, we introduce two methods to efficiently compute the LGF.

\section{Fast evaluation and compact support at large $c^2$}
\label{sec:Large_c}
In this section, we derive a series expansion for $B_c$. From the expansion, we can obtain two unique properties of the LGF of the screened Poisson equation: (a) exponential convergence in series expansion, and (b) exponential decay in $|m|+|n|$. We will later show the duality between them. These two properties give a fast method to evaluate the LGF at relatively large $c^2$ and a fast way to solve the screened Poisson equation by applying the LGF. 
\subsection{Series expansion}
Recall the LGF, $B_c$, at an arbitrary point $\boldsymbol{n}$ can be written as:
\begin{equation}
    B_c(\boldsymbol{n}) = \frac{1}{(2\pi)^2}\int_{I^2} e^{-i\boldsymbol{n}\cdot \boldsymbol{\xi}} \frac{1}{\sigma (\boldsymbol{\xi}) + c^2}2\boldsymbol{\xi},
    \label{eq:HelmLGF_Recall}
\end{equation}
where the function $\sigma(\boldsymbol{\xi})$ reads
\begin{equation}
    \sigma(\boldsymbol{\xi}) = 2\alpha_1 - 2\alpha_1\cos(\xi_1) + 2 - 2\cos(\xi_2).
\end{equation}
We define
\begin{equation}
    \rho(\boldsymbol{\xi}) =  2\alpha_1\cos(\xi_1) + 2\cos(\xi_2), \quad \lambda = 2 + 2\alpha_1.
\end{equation}
Thus, we can write:
\begin{equation}
    \sigma(\boldsymbol{\xi}) = \lambda - \rho(\boldsymbol{\xi}),
\end{equation}
and thus we can write \cref{eq:int_2d_form} as
\begin{equation}
    B_c(\boldsymbol{n}) = \frac{1}{(2\pi)^2}\int_{I^2} e^{-i\boldsymbol{n}\cdot \boldsymbol{\xi}} \frac{1}{\lambda - \rho (\boldsymbol{\xi}) + c^2}d\boldsymbol{\xi}.
    \label{eq:HelmLGF_rho}
\end{equation}
Now since
\begin{equation}
    \rho(\boldsymbol{\xi}) \in [-\lambda,\lambda] \Rightarrow |\rho(\boldsymbol{\xi})| < \lambda+c^2,
\end{equation}
and $c^2 > 0$, we can expand the integral formally as
\begin{equation}
    \begin{aligned}
    B_c(\boldsymbol{n}) 
    &= \frac{1}{(2\pi)^2}\int_{I^2} e^{-i\boldsymbol{n}\cdot \boldsymbol{\xi}} \frac{1}{\lambda - \rho (\boldsymbol{\xi}) + c^2}d\boldsymbol{\xi} \\
    &= \frac{1}{(2\pi)^2}\int_{I^2} e^{-i\boldsymbol{n}\cdot \boldsymbol{\xi}} \frac{1}{\lambda+c^2}\frac{1}{1 - \rho (\boldsymbol{\xi})/(\lambda+c^2)}d\boldsymbol{\xi} \\
    &= \frac{1}{(2\pi)^2}\int_{I^2} e^{-i\boldsymbol{n}\cdot \boldsymbol{\xi}} \frac{1}{\lambda+c^2} \sum\limits_{k = 0}^\infty \left(\frac{\rho( \boldsymbol{\xi})}{\lambda+c^2}\right)^k d\boldsymbol{\xi} \\
    &= \frac{1}{(2\pi)^2} \frac{1}{\lambda+c^2} \sum\limits_{k = 0}^\infty\int_{I^2} e^{-i\boldsymbol{n}\cdot \boldsymbol{\xi}} \left(\frac{\rho(\boldsymbol{\xi})}{\lambda+c^2}\right)^k d\boldsymbol{\xi} \\
    &= \frac{1}{(2\pi)^2} \frac{1}{\lambda+c^2} \sum\limits_{k = 0}^\infty \left(\frac{\lambda}{\lambda+c^2}\right)^k \int_{I^2} e^{-i\boldsymbol{n}\cdot \boldsymbol{\xi}} \left(\frac{\rho(\boldsymbol{\xi})}{\lambda}\right)^k d\boldsymbol{\xi} 
    \end{aligned}
    \label{eq:deriveExp}
\end{equation}
To show that this series converges, it suffices to show that the dominated convergence theorem applies. That is, as long as we can find an integrable function that dominates the sequence of the integrand, the equation above holds, which leads to the following lemma:
\begin{lemma}
    Fix $\boldsymbol{n} \in \mathbb{Z}^2$. Define 
    \begin{equation}
        f_k(\boldsymbol{\xi}) = e^{-i\boldsymbol{n}\cdot \boldsymbol{\xi}} \frac{1}{\lambda+c^2} \sum\limits_{l = 0}^k \left(\frac{\rho(\boldsymbol{\xi})}{\lambda+c^2}\right)^l.
    \end{equation}
    Then $f_k$ is dominated by a constant:
    \begin{equation}
        |f_k| \leq \frac{1}{c^2}.
    \end{equation}
    And since the integration domain $I^2$ is finite, a constant function is integrable.
\end{lemma}
\begin{proof}
Consider
\begin{equation}
\begin{aligned}
    |f_k| &= \left\lvert e^{-i\boldsymbol{n}\cdot \boldsymbol{\xi}} \frac{1}{\lambda+c^2} \sum\limits_{l = 0}^k \left(\frac{\rho( \boldsymbol{\xi})}{\lambda+c^2}\right)^l \right\lvert \\
    &\leq |e^{-i\boldsymbol{n}\cdot \boldsymbol{\xi}}| \left\lvert\frac{1}{\lambda+c^2} \right\lvert \left\lvert  \sum\limits_{l = 0}^k \left(\frac{\lambda}{\lambda+c^2}\right)^l \left(\frac{\rho(\boldsymbol{\xi})}{\lambda}\right)^l\right\lvert \\
    &\leq \left\lvert\frac{1}{\lambda+c^2} \right\lvert \left\lvert  \sum\limits_{l = 0}^k \left(\frac{\lambda}{\lambda+c^2}\right)^l \right\lvert \\
    &\leq \left\lvert\frac{1}{\lambda+c^2} \right\lvert \frac{1}{1 - \lambda/(\lambda+c^2)} = \frac{1}{c^2} .
\end{aligned}
\end{equation}
Thus, $f_k$ is dominated by $1/c^2$.
\end{proof}
With \cref{eq:deriveExp}, we can define an approximation of $B_c$.
\begin{definition}
The $N$-term approximation of $B_c(\boldsymbol{n})$, denoted as $G_N(c, \boldsymbol{n})$, is defined as
\begin{equation}
    G_N(c, \boldsymbol{n}) = \frac{1}{(2\pi)^2} \frac{1}{\lambda+c^2} \sum\limits_{k = 0}^{N-1}  \left(\frac{\lambda}{\lambda+c^2}\right)^k \int_{I^2} e^{-i\boldsymbol{n}\cdot \boldsymbol{\xi}} \left(\frac{\rho( \boldsymbol{\xi})}{\lambda}\right)^k d\boldsymbol{\xi} .
    \label{eq:G_N_approx}
\end{equation}
\end{definition}
We can bound the error of this $N$-term approximation with the following theorem.
\begin{theorem}
    The truncation error from the $N$-term approximation of $B_c(\boldsymbol{n})$ is bounded by
    \begin{equation}
        |\epsilon_N(\boldsymbol{n})| = |B_c(\boldsymbol{n}) - G_N(c,\boldsymbol{n})| \leq \frac{1}{c^2}\left(\frac{\lambda}{\lambda+c^2}\right)^N \quad \forall \boldsymbol{n} \in \mathbb{Z}^2.
    \end{equation}
    \label{thm:asym_bound}
\end{theorem}
\begin{proof}
Fix $\boldsymbol{n} \in \mathbb{Z}^n$, and the truncation error is:
\begin{equation}
    \begin{aligned}
    \epsilon_N
    &= B_c(\boldsymbol{n}) - G_N(c, \boldsymbol{n}) \\
    &= \frac{1}{(2\pi)^2} \frac{1}{\lambda+c^2} \sum\limits_{k = N}^\infty \left(\frac{\lambda}{\lambda+c^2}\right)^k \int_{I^2} e^{-i\boldsymbol{n}\cdot \boldsymbol{\xi}} \left(\frac{\rho(\boldsymbol{\xi})}{\lambda}\right)^k d\boldsymbol{\xi} .
    \end{aligned}
\end{equation}
The error, $\epsilon_N$, can be bounded by the following:
\begin{equation}
    \begin{aligned}
    |\epsilon_N|
    &= \left\lvert\frac{1}{(2\pi)^2} \frac{1}{\lambda+c^2} \sum\limits_{k = N}^\infty \left(\frac{\lambda}{\lambda+c^2}\right)^k \int_{I^2} e^{-i\boldsymbol{n}\cdot \boldsymbol{\xi}} \left(\frac{\rho(\boldsymbol{\xi})}{\lambda}\right)^{k+N} d\boldsymbol{\xi} \right\lvert \\
    &\leq \left\lvert\frac{1}{(2\pi)^2}\frac{1}{\lambda+c^2}\left(\frac{\lambda}{\lambda+c^2}\right)^N \sum\limits_{k = 0}^\infty 4\pi^2\left(\frac{\lambda}{\lambda+c^2}\right)^k\right\lvert \\
    &= \frac{1}{\lambda+c^2}\left(\frac{\lambda}{\lambda+c^2}\right)^N \frac{1}{1 - \lambda/(\lambda+c^2)} \\
    &= \frac{1}{c^2}\left(\frac{\lambda}{\lambda+c^2}\right)^N .
    \end{aligned}
\end{equation}
Since this error bound is independent of $\boldsymbol{n}$, it is true for all $\boldsymbol{n}$.
\end{proof}
\subsection{Analytical expressions of the $N$-term approximation}
It turns out that each term in the series representation of $B_c$ can be analytically computed as functions of multinomial coefficients. To do so, we define 
\begin{equation}
    g_k(\boldsymbol{n}) = \frac{1}{4\pi^2}\int_{I^2} e^{-i\boldsymbol{n}\cdot \boldsymbol{\xi}} \left(\frac{\rho(\boldsymbol{\xi})}{\lambda}\right)^k d\boldsymbol{\xi},
\end{equation}
so that we can write
\begin{equation}
    G_N(c, \boldsymbol{n}) = \frac{1}{\lambda+c^2} \sum\limits_{k = 0}^{N-1}  \left(\frac{\lambda}{\lambda+c^2}\right)^k g_k(\boldsymbol{n}).
\end{equation}
Now we can focus on those $g_k(\boldsymbol{n})$ terms alone. The following theorem gives an analytical expression of $g_k(\boldsymbol{n})$.
\begin{theorem}
    The function $g_k(n,m)$ is nonzero if and only if $k \geq |n| + |m|$ and $k - |n| - |m|$ is even. In that case, 
    \begin{equation}
        g_k(n,m) 
    = \frac{1}{\lambda^k}\sum\limits_{l=0}^{(k - n - m)/2} 
    \alpha_1^{n + 2l}\begin{pmatrix}
    k\\
    l, n+l, (k-n-2l - m)/2, (k-n-2l + m)/2
    \end{pmatrix}
    \label{eq:g_k_multinomial}
    \end{equation}
    where
    \begin{equation}
        \begin{pmatrix}
    k\\
    a, b, c, d
    \end{pmatrix} = \frac{k!}{a! b! c! d!}
    \end{equation}
    is the multinomial coefficient.
\label{thm:g_k_multinomial}
\end{theorem}
\begin{remark}
    A way to effectively evaluate \cref{eq:g_k_multinomial} without numerical overflow is using the log Gamma function, which is relatively well-behaved. With this result, for a finite $k$, one can efficiently evaluate $g_k$ directly using the built-in \verb!lgamma! function in C++ or using existing numerical packages such as \verb|Boost| \cite{schaling2011boost} and \verb|SciPy| \cite{2020SciPy-NMeth}. 
\end{remark}
\begin{remark}
    This theorem completes the entire asymptotic expansion of $B_c$. This expression can also be derived from the perspective of a random walk with killing. Similar results for the LGF on square lattices ($\alpha_1 = 1$) have been derived using a random walk argument \cite{michta2021asymptotic}.
\end{remark}
\begin{proof}[Proof of \cref{thm:g_k_multinomial}]
    We directly expand the integral form of $g_k$
    \begin{equation}
        \begin{aligned}
        g_k(\boldsymbol{n}) 
        &= \frac{1}{4\pi^2}\int_{I^2} e^{-i\boldsymbol{n}\cdot \boldsymbol{\xi}} \left(\frac{\rho(\boldsymbol{\xi})}{\lambda}\right)^k d\boldsymbol{\xi} \\
        &= \frac{1}{4\pi^2}\frac{1}{\lambda^k}\int_{I^2} e^{-i\boldsymbol{n}\cdot \boldsymbol{\xi}} \left(2\alpha_1 \cos(\xi_1) +2\cos(\xi_2) \right)^k d\boldsymbol{\xi} \\
        &= \frac{1}{4\pi^2}\frac{1}{\lambda^k}\int_{I^2} e^{-i\boldsymbol{n}\cdot \boldsymbol{\xi}} \sum\limits_{l = 0}^k \begin{pmatrix}
        k\\
        l
        \end{pmatrix} (2\alpha_1 \cos(\xi_1))^l (2\cos(\xi_2) )^{k - l} d\boldsymbol{\xi} \\
        &= \frac{1}{4\pi^2}\frac{1}{\lambda^k}\sum\limits_{l = 0}^k \begin{pmatrix}
        k\\
        l
        \end{pmatrix} \alpha_1^l \int_{I} e^{-in_1 \xi_1}(2 \cos(\xi_1))^l d\xi_1  \int_{I} e^{-in_2 \xi_2}(2\cos(\xi_2) )^{k - l} d\xi_2.
        \end{aligned}
        \label{eq:int_g_k_proof}
    \end{equation}
    A direct calculation shows that 
    \begin{equation}
        \int_{I} e^{-in \xi}(2 \cos(\xi))^p d\xi = \begin{cases}
        2\pi\begin{pmatrix}
        p\\
        (p - n)/2
        \end{pmatrix} &\textrm{if } (p - n) \geq 0 \textrm{ and $(p - n)$ is even} \\ 
        0 &\textrm{otherwise}.
        \end{cases}
    \end{equation}
    Plugging this expression into \cref{eq:int_g_k_proof}, we obtain the desired result.
\end{proof}
In \cref{fig:err_conv_asym}, we demonstrate the error convergence rate of the LGF approximated by $G_N$ at selected values of $(n,m)$ at $\alpha_1 = 0.75$ compared with corresponding error bounds given by \cref{thm:asym_bound}. 
\begin{figure}
    \centering
    \includegraphics[width=0.95\textwidth]{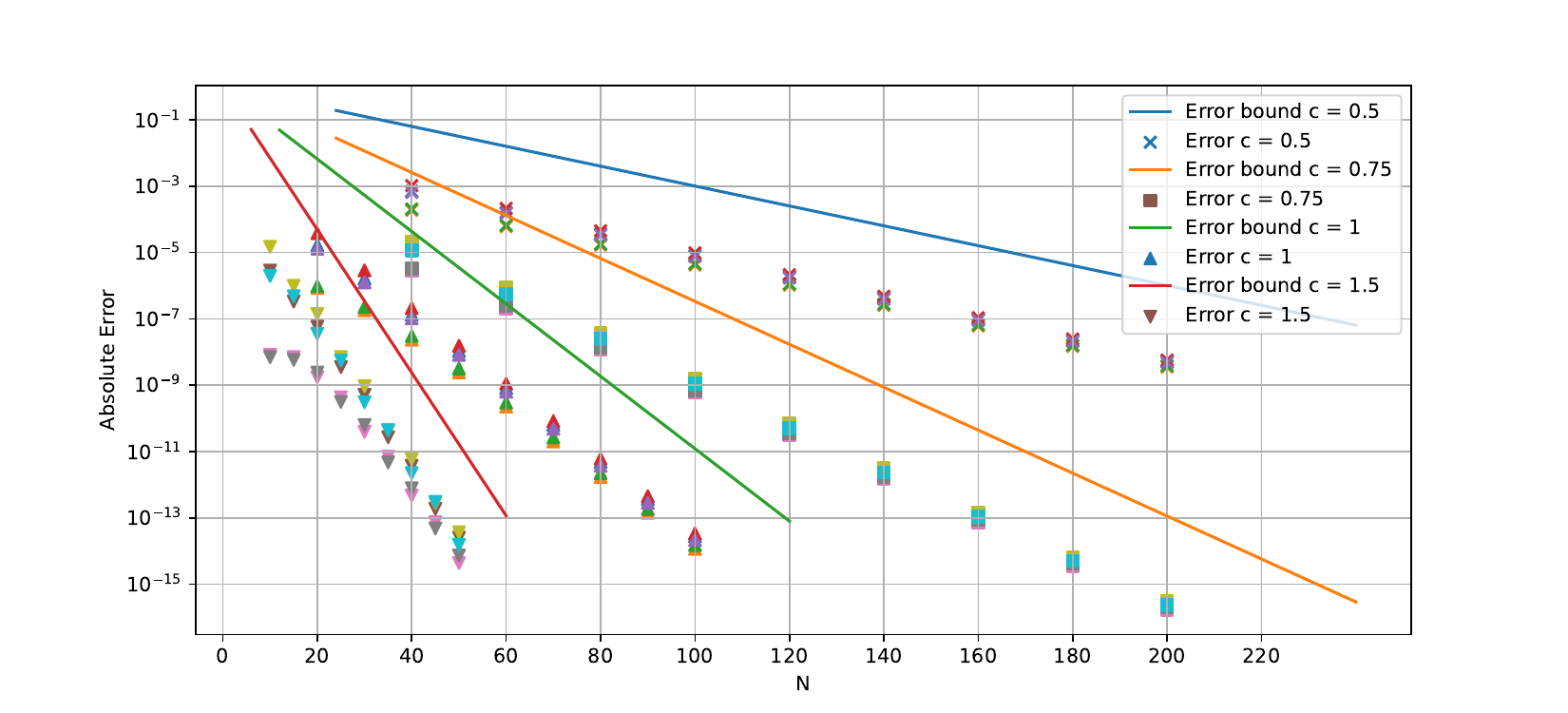}
    \caption{Error of $G_N$ for various $c$ at $\alpha_1 = 0.75$. We randomly choose 5 points within the square $[0, 10)^2$ and evaluate their $G_N$ approximation using various $N$ at different $c$. We compare the resulting $G_N$ with the solution obtained by evaluating $B_c$ at those points using \cref{eq:appell_LGF}. We also show the error bounds given by \cref{thm:asym_bound}. }
    \label{fig:err_conv_asym}
\end{figure}

\subsection{Spatial decay of the lattice Green's function}
Fix $n, m \geq 0$. With $\boldsymbol{n} = (n,m)$, 
\begin{equation}
    G_{n+m}(c, \boldsymbol{n}) = \frac{1}{\lambda+c^2} \sum\limits_{k = 0}^{n+m-1}  (\frac{\lambda}{\lambda+c^2})^k g_k(\boldsymbol{n}).
\end{equation}
By \cref{thm:g_k_multinomial}, $g_k(n,m) = 0$ for $k < n+m$. As a result, $G_{n+m}(c,\boldsymbol{n}) = 0$. Thus, by \cref{thm:asym_bound}, we can write
\begin{equation}
    |B_c(\boldsymbol{n})| = |B_c(\boldsymbol{n}) - G_{n+m}(c,\boldsymbol{n})| = |\epsilon_{n+m}| \leq \frac{1}{c^2}\left(\frac{\lambda}{\lambda+c^2}\right)^{(n+m)}
\end{equation}
Thus, $B_c(n,m)$ decays exponentially fast as $n + m$ increases.

All through this result is immediate from \cref{thm:asym_bound,thm:g_k_multinomial}, it has at least two important implications.

First, when approximating $B_c$ on a large domain using $G_N$, the number of terms in $G_N$ we need to evaluate decreases as $m+n$ increases. By only evaluating those nonzero terms, we can achieve significant computational savings when evaluating each term of $B_c$.

Second, when solving the screened Poisson equation using $B_c$ to a certain precision, we need only convolve $B_c$ within regions where $B_c$ is larger than the prescribed precision. In this way, applying $B_c$ can be made much more computationally efficient.

As it is evident from the error estimates in \cref{thm:asym_bound}, when $c^2$ is large, one can approximate $B_c$ to the machine precision using only a few terms, and we only need to evaluate a small number of $B_c$ since it decays exponentially fast. As a result, when $c^2$ is large, it is favorable to evaluate the LGF using \cref{eq:G_N_approx}.

\section{Calculation of the lattice Green's function at arbitrary nonzero $c^2$}
\label{sec:LGFcomp}
For smaller values of $c^2$, evaluating $B_c$ using \cref{eq:appell_LGF,eq:G_N_approx} becomes more expensive. To resolve this problem, we introduce a one-dimensional integral formulation of $B_c$ through the following theorem:
\begin{theorem}
    The value of $B_c(n,m)$ for any $n,m \in \mathbb{Z}$ and $c > 0$ can be written as:
    \begin{equation}
    {B}_c(n,m) = \frac{1}{2\pi}\int^{\pi}_{-\pi} \frac{e^{i\theta n}}{K^{|m|}} \frac{d\theta}{K - 1/K}
    \end{equation}
    where
    \begin{equation}
        K = \frac{\phi + \sqrt{\phi^2 - 4}}{2}, \quad \phi = \lambda + c^2 - 2\alpha_1\cos(\theta), \quad \lambda = 2+2\alpha_1.
        \label{eq:integral_LGF_1D} 
    \end{equation}
    \label{thm:LGF_int_1D}
\end{theorem}
\begin{proof}
We first rewrite the governing equation (\cref{eq:def_LGF} without the far field boundary condition) as
\begin{multline}
    \lambda B_c(n,m) - (B_c(n, m-1) + B_c(n, m+1)) \\
    = \alpha_1(B_c(n - 1,m) + B_c(n + 1,m)) - c^2 B_c(n,m) + \delta_{0n}\delta_{0m}.
\end{multline}
    The $N$-term discrete Fourier transform in the first coordinate is
    \begin{equation}
        \tilde{B}_c^N (k,m) = \sum\limits_{n = 1 - N/2}^{N/2} B_c(n,m)e^{-2\pi i k n/N}.
    \end{equation}
    We first impose that ${B}_c (n,m)$ is periodic in $n$ with periodicity $N$ (assuming that $N$ is even); we later relax periodicity in the limit $N \rightarrow \infty$.  Consider the following:
    \begin{align*}
        &\tilde{B}^N_c(k,m-1) + \tilde{B} ^N_c(k,m+1) - \lambda\tilde{B}^N_c(k,m) \\
        = &\sum\limits_{n = 1 - N/2}^{N/2} \left[ B_c(n,m - 1)e^{-2\pi i k n/N} + B_c(n,m + 1)e^{-2\pi i k n/N} - \lambda B_c(n,m)e^{-2\pi i k n/N} \right] \\
        = &\sum\limits_{n = 1 - N/2}^{N/2} \left[ B_c(n,m - 1) + B_c(n,m + 1) - \lambda B_c(n,m)\right]e^{-2\pi i k n/N} \\
        = &\sum\limits_{n = 1 - N/2}^{N/2} \left[ c^2 B_c(n,m) -\alpha_1 B_c(n - 1,m) - \alpha_1 B_c(n + 1,m) -\delta_{0m}\delta_{0n} \right]e^{-2\pi i k n/N}.
    \end{align*}
    With the periodicity assumption, we can write
    \begin{align*}
        &\sum\limits_{n = 1 - N/2}^{N/2} \left[ c^2 B_c(n,m) -\alpha_1 B_c(n - 1,m) - \alpha_1 B_c(n + 1,m)\right]e^{-2\pi i k n/N} \\
        =&\sum\limits_{n = 1 - N/2}^{N/2}  B_c(n ,m)\left[c^2e^{-2\pi i k n/N} - \alpha_1 e^{-2\pi i k (n + 1)/N} - \alpha_1 e^{-2\pi i k (n - 1)/N}\right] \\
        =&\sum\limits_{n = 1 - N/2}^{N/2}  B_c(n ,m)e^{-2\pi i k n/N}(c^2 - 2 \alpha_1 \cos(2\pi k /N)) \\
        =&[c^2 - 2 \alpha_1 \cos(2\pi k /N)] \tilde{B}^N_c(k,m)
    \end{align*}
    and
    \begin{equation}
        \sum\limits_{n = 1 - N/2}^{N/2} \delta_{0m}\delta_{0n} e^{-2\pi i k n/N} = \delta_{0m}.
    \end{equation}
    As a result, we have
    \begin{equation}
        \tilde{B}^N_c(k,m-1) + \tilde{B}^N_c(k,m+1) = [\lambda+c^2 - 2 \alpha_1 \cos(2\pi k /N)] \tilde{B}^N_c(k,m) - \delta_{0m}.
    \end{equation}
    On the one hand, if $m \neq 0$, we have
    \begin{equation}
        \tilde{B}^N_c(k,m-1) + \tilde{B}^N_c(k,m+1) + [2 \alpha_1 \cos(2\pi k /N) - \lambda - c^2]\tilde{B}^N_c(k,m) = 0.
    \end{equation}
    This type of recurrence relation can be solved by assuming the following ansatz \cite{buneman1971analytic}:
    \begin{equation}
        \tilde{B}^N_c(k,m) = \tilde{B}^N_c(k,0)/K^{|m|}.
    \end{equation}
    By directly plugging in our recurrence relation, $K$ can be solved using
    \begin{equation}
        K + 1/K = \lambda + c^2 - 2 \alpha_1 \cos(2\pi k /N) := \phi_N.
    \end{equation}
    To enforce the condition that $B_c(n,m) \rightarrow 0$ as $|m| \rightarrow \infty$, we need $|K| > 1$. Thus, the appropriate quadratic root is
    \begin{equation}
        K = \frac{\phi_N + \sqrt{\phi_N^2 - 4}}{2}.
    \end{equation}
    On the other hand, if $m = 0$, we have
    \begin{equation}
        \tilde{B}^N_c(k,-1) + \tilde{B}^N_c(k,1) + [2 \alpha_1 \cos(2\pi k /N) - \lambda - c^2]\tilde{B}^N_c(k,0) = -1.
    \end{equation}
    Substituting the equation of $K$, we obtain the solution of $\tilde{B}^N_c(k,0)$ as
    \begin{equation}
        2\tilde{B}^N_c(k,0)/K - \phi_N\tilde{B}^N_c(k,0) = -1  \Rightarrow \tilde{B}^N_c(k,0) = \frac{1}{K - 1/K}.
    \end{equation}
    Thus, the expression of $\tilde{B}^N_c(k,m)$ comes out to be
    \begin{equation}
        \tilde{B}^N_c(k,m) = \frac{1}{K - 1/K}\frac{1}{K^{|m|}}.
    \end{equation}
    With the expression of $\tilde{B}^N_c(k,m)$, we take the inverse discrete Fourier transform to obtain $B_c(n,m)$
    \begin{equation}
        B_c(n,m) = \frac{1}{N}\sum\limits_{k = 1 - N/2}^{N/2} e^{2\pi i kn/N} \tilde{B}^N_c(k,m) = \frac{1}{N}\sum\limits_{k = 1 - N/2}^{N/2} e^{2\pi i kn/N} \frac{1}{K - 1/K}\frac{1}{K^{|m|}}.
    \end{equation}
    Now, we are ready to take $N$ to infinity. To do so, define
    \begin{equation}
        \theta_k = 2\pi k/N, \quad \Delta \theta = 2\pi /N.
    \end{equation}
    The expression of $B_c(n,m)$ becomes
    \begin{equation}
        B_c(n,m) = \Delta \theta \sum\limits_{k = 1 - N/2}^{N/2} e^{in\theta_k} \frac{1}{K - 1/K}\frac{1}{K^{|m|}}
    \end{equation}
    where
    \begin{equation}
        K = \frac{\phi_N + \sqrt{\phi_N^2 - 4}}{2}, \quad \phi_N = \lambda + c^2 - 2 \alpha_1 \cos(\theta_k).
    \end{equation}
    Note that now the summands are composed entirely of $\theta_k$, without direct involvement of $N$. By taking $N$ to infinity, we are creating a Riemann sum. Since the function defining the summands is bounded and continuous, the Riemann sum converges, and we write 
    \begin{equation}
        \lim\limits_{N \rightarrow \infty} \Delta \theta \sum\limits_{k = 1 - N/2}^{N/2} e^{in\theta_k} \frac{1}{K - 1/K}\frac{1}{K^{|m|}} = \frac{1}{2\pi}\int_{-\pi}^{\pi} e^{in\theta} \frac{1}{K - 1/K}\frac{d\theta}{K^{|m|}}
    \end{equation}
    where
    \begin{equation}
        K = \frac{\phi + \sqrt{\phi^2 - 4}}{2}, \quad \phi = \lambda + c^2 - 2 \alpha_1 \cos(\theta).
    \end{equation}
    As a result, can write
    \begin{equation}
        B_c(n,m) = \frac{1}{2\pi}\int_{-\pi}^{\pi} \frac{ e^{in\theta}}{K - 1/K}\frac{d\theta}{K^{|m|}}.
    \end{equation}
    Since we have taken $N$ to infinity, $B_c(n,m)$ does not have to be periodic anymore.
\end{proof}
\begin{remark}
    The proof of \cref{thm:LGF_int_1D} generalizes the techniques presented in \cite{buneman1971analytic}, where the author only considered the case of $c^2 = 0$ and $\alpha_1 = 1$. In that case, $K$ only has one root, so there is no need to identify the correct root. 
    
    Using similar techniques, one can show that, when $c^2 = 0$ and $\alpha_1 < 1$, the corresponding LGF is
    \begin{equation}
        B_0(n,m) - B_0(0,0) = \frac{1}{2\pi}\int_{-\pi}^{\pi} \left(\frac{e^{in\theta}}{K^{|m|}} - 1\right)\frac{d\theta}{K - 1/K},
    \end{equation}
    where
    \begin{equation}
        K = \frac{\phi + \sqrt{\phi^2 - 4}}{2}, \quad \phi = \lambda - 2 \alpha_1 \cos(\theta).
    \end{equation}
    The proof follows from the proof of \cref{thm:LGF_int_1D}.
\end{remark}
\begin{remark}
In contrast to \cref{eq:bessel_representation2}, the integral in \cref{eq:integral_LGF_1D} has a finite integration domain and an integrand consisting of elementary functions only. As a result, numerical integrating \cref{eq:integral_LGF_1D} is faster and more stable.
\end{remark}

\section{Convergence rate of the trapezoidal rule approximation}
\label{sec:TrapConv}
We cannot reduce the one-dimensional integral presentation of $B_c(n,m)$ any further. Thus, we need to evaluate it numerically. It turns out, however, that the trapezoidal rule approximation yields an exponential convergence rate for this particular integral. To show this, we first invoke the following general theorem on the convergence rate of trapezoidal rule approximations \cite{trefethen2014exponentially}:
\begin{theorem}
(Trefethen and Weideman \cite{trefethen2014exponentially}) Let
\begin{equation}
    I = \int_{-\pi}^{\pi}v(\theta)d\theta.
\end{equation}
For any positive integer $N$, define the trapezoidal rule approximation:
\begin{equation}
    I_N = \frac{2\pi}{N}\sum\limits_{k=1}^N v(\theta_k)
    \label{eq:trap_appr}
\end{equation}
where $\theta_k = 2\pi k/N - \pi$. Suppose $v$ is $2\pi$ periodic and analytic and satisfies $|v(\theta)| < M$ in the strip $-\gamma < \Im(\theta) < \gamma$ for some $\gamma > 0$. Then for any $N \geq 1$,
    \begin{equation}
        |I_N - I| \leq \frac{4\pi M}{e^{\gamma N} - 1},
        \label{eq:converate_Tref}
    \end{equation}
    and the constant $4\pi$ is as small as possible.
    \label{thm:spconv_1d}
\end{theorem}
Using the above theorem, we can show the following result regarding the convergence rate of the trapezoidal rule approximation for \cref{eq:integral_LGF_1D}.
\begin{theorem}
    Let $\alpha_1 \in (0,1)$ and
    \begin{equation}
        v(\theta) = \frac{e^{i\theta n}}{K^{|m|}} \frac{1}{K - 1/K}, \quad \phi = \lambda + c^2 - 2\alpha_1\cos(\theta), \quad K = \frac{\phi + \sqrt{\phi^2 - 4}}{2}.
    \end{equation}
    Let
    \begin{equation}
        I = \int_{-\pi}^{\pi} v(\theta) d\theta
    \end{equation}
    and $I_N$ be its trapezoidal rule approximation. Then for any positive real number, $\gamma_c$, satisfying
    \begin{equation}
    \gamma_c < \log \left(1 + \frac{c^2}{2\alpha_1} + \sqrt{\left(1+\frac{c^2}{2\alpha_1}\right)^2 - 1} \right),
    \end{equation}
    for any $N \leq 1$, we have
    \begin{equation}
        |I_N - I| \leq \frac{4\pi M}{e^{\gamma_c N} - 1},
    \end{equation}
    where $M$ is
    \begin{equation}
        M = \sup\limits_{|\Im(\theta)| < \gamma_c} \left| \frac{e^{i\theta n}}{K^{|m|}}\frac{1}{K - 1/K} \right|.
        \label{eq:err_M}
    \end{equation}
    \label{thm:sp_conv_lgf}
\end{theorem}
Before proving the above theorem, we first prove the following technical lemma
\begin{lemma}
    \begin{equation}
        |\alpha_1 \cos(\theta)| <  \lambda/2 + c^2/2 - 1 \Rightarrow \phi^2 - 4 \notin \mathbb{R}_{<0}.
    \end{equation}
    \label{lm:nonReal}
\end{lemma}
\begin{proof}[Proof of \cref{lm:nonReal}]
    We prove this lemma by contradiction. Assume 
    \begin{equation}
        \phi^2 - 4 = v \in \mathbb{R}_{<0}.
    \end{equation}
    We rewrite this equation as:
    \begin{equation}
        (\lambda + c^2 - 2\alpha_1\cos(\theta))^2 = 4+v
    \end{equation}
    We solve for $\alpha_1\cos(\theta)$:
    \begin{equation}
        \alpha_1\cos(\theta) = \frac{1}{2}\left[\lambda+c^2\pm\sqrt{4+v}\right].
    \end{equation}
    Then we estimate
    \begin{equation}
        |\alpha_1\cos(\theta)| \geq \frac{1}{2}(\lambda+c^2) - \frac{1}{2}|\Re(\sqrt{4+v})| > \frac{1}{2}(\lambda+c^2) - 1.
    \end{equation}
    Thus, we have a contradiction.
\end{proof}
\begin{proof}[Proof of \cref{thm:sp_conv_lgf}]
    To use \cref{thm:spconv_1d}, we need to find a strip within which our specific $v(\theta)$ is analytic
\begin{equation}
    v(\theta) = \frac{e^{i\theta n}}{K^{|m|}} \frac{1}{K - 1/K}, \quad \phi = \lambda + c^2 - 2\alpha_1\cos(\theta), \quad K = \frac{\phi + \sqrt{\phi^2 - 4}}{2}.
\end{equation}
Inspecting the above expression, we know that $v(\theta)$ is analytic in a strip if, in which, $K - 1/K \neq 0$ and $\phi^2 - 4 \notin \mathbb{R}_{<0}$.

We first find a strip in which $K - 1/K \neq 0$. To do so, we only need to pick any finite $\gamma$ such that 
\begin{equation}
    K - 1/K \neq 0 \quad \forall \theta \in \mathbb{C}: |\Im(\theta)| < \gamma.
\end{equation}
We have 
\begin{equation}
    K - 1/K = \sqrt{\phi^2 - 4},
\end{equation}
so 
\begin{equation}
    K - 1/K = 0 \iff \phi^2 = 4 \iff \phi = \pm 2.
\end{equation}
Directly plugging in the expression of $\phi$, we obtain:
\begin{equation}
    \cos(\theta) = \frac{1}{\alpha_1}\left[\frac{\lambda}{2} \pm 1 + \frac{1}{2}c^2\right].
    \label{eq:cos_eq}
\end{equation}
To facilitate the discussion, we denote the two possible values on the RHS as:
\begin{equation}
    \phi^+_m = \frac{1}{\alpha_1}\left[\frac{\lambda}{2} + 1 + \frac{1}{2}c^2\right], \quad \phi^-_m = \frac{1}{\alpha_1}\left[\frac{\lambda}{2} - 1 + \frac{1}{2}c^2\right].
\end{equation}
A necessary condition for \cref{eq:cos_eq} to be satisfied is 
\begin{equation}
    \exp({\Im(\theta)}) = \phi^+_m \pm \sqrt{(\phi^+_m)^2 - 1} \textrm{ or } \exp{(\Im(\theta))} = \phi^-_m \pm \sqrt{(\phi^-_m)^2 - 1}
\end{equation}
In addition, we notice that
\begin{equation}
    \frac{1}{\phi_m^\pm + \sqrt{(\phi^\pm_m)^2 - 1}} = \frac{\phi_m^\pm - \sqrt{(\phi^\pm_m)^2 - 1}}{(\phi_m^\pm + \sqrt{(\phi^\pm_m)^2 - 1})(\phi_m^\pm - \sqrt{(\phi^\pm_m)^2 - 1})} = \phi_m^\pm - \sqrt{(\phi^\pm_m)^2 - 1}.
\end{equation}
Consequently,
\begin{equation}
    \log(\phi^\pm_m + \sqrt{(\phi^\pm_m)^2 - 1}) = -\log(\phi^\pm_m - \sqrt{(\phi^\pm_m)^2 - 1}).
\end{equation}
Since the logarithm function is a monotonically increasing function and that $|\log(\phi^\pm_m + \sqrt{(\phi^\pm_m)^2 - 1})| = |\log(\phi^\pm_m - \sqrt{(\phi^\pm_m)^2 - 1})|$, to ensure analyticity within the strip, we need 
\begin{equation}
    |\Im(\theta)| < \gamma := \log(\phi^-_m + \sqrt{(\phi^-_m)^2 - 1}).
\end{equation}
We then focus on the second condition regarding $K$, i.e.
\begin{equation}
    \phi^2 - 4 \notin \mathbb{R}_{<0}.
\end{equation}
Since the function $e^x + e^{-x}$ monotonically increases with $x$ when $x \geq 0$, within the strip of $|\Im(\theta)| < \gamma$, we have
\begin{equation}
\begin{aligned}
    |\cos(\theta)| &\leq |e^{\Im(\theta)}|/2 + |e^{-\Im(\theta)}|/2 < \frac{1}{2}\left(e^\gamma + e^{-\gamma}\right) \\
    &= \frac{1}{2}\left(\phi^-_m + \sqrt{(\phi^-_m)^2 - 1} + \phi^-_m - \sqrt{(\phi^-_m)^2 - 1}\right) = \phi^-_m.
\end{aligned}
\end{equation}
Thus, $|\alpha_1\cos(\theta)| < (\lambda/2 - 1 + c^2/2)$, so by \cref{lm:nonReal}, $\phi^2 - 4 \notin \mathbb{R}_{<0}$. Considering $K - 1/K \neq 0$, we know that $v(\theta)$ is analytic within the strip.
If we pick any $\gamma_c < \gamma$, we have
\begin{equation}
    |v(\theta)| \leq M = \sup\limits_{|\Im(\theta)| < \gamma_c} \left| \frac{e^{i\theta n}}{K^{|m|}}\frac{1}{K - 1/K} \right|.
\end{equation}
We write
\begin{equation}
    \phi^-_m = \frac{1}{\alpha_1}\left[\frac{\lambda}{2} - 1 + \frac{1}{2}c^2\right] = \frac{1}{\alpha_1}\left[1+\alpha_1 - 1 + \frac{1}{2}c^2\right] = 1 + \frac{c^2}{2\alpha_1}.
\end{equation}
Thus, we can find $\gamma$ to be:
\begin{equation}
    \gamma = \log\left(1 + \frac{c^2}{2\alpha_1} + \sqrt{\left(1 + \frac{c^2}{2\alpha_1}\right)^2 - 1}\right).
\end{equation}
And by \cref{thm:spconv_1d}, we conclude with the desired result.
\end{proof}
\begin{remark}
    Investigating \cref{thm:sp_conv_lgf} further, we can understand the effect of $n$ in the error in \cref{eq:err_M} better. We can write
    \begin{equation}
        M = \sup\limits_{|\Im(\theta)| < \gamma_c} \left| \frac{e^{i\theta n}}{K^{|m|}}\frac{1}{K - 1/K} \right| \leq e^{\gamma_c n}\left[\sup\limits_{|\Im(\theta)| < \gamma_c}  \left|\frac{1}{K^{|m|}}\frac{1}{K - 1/K} \right|\right].
    \end{equation}
    Then the error estimate becomes
    \begin{equation}
        |I_N - I| \leq \frac{4\pi e^{\gamma_c n} \overline{M}}{e^{\gamma_c N} - 1} = \frac{4\pi \overline{M}}{e^{\gamma_c (N - n)} - e^{-\gamma_c n}} \leq \frac{4\pi \overline{M}}{e^{\gamma_c (N - n)} - 1}
    \end{equation}
    where
    \begin{equation}
        \overline{M} = \sup\limits_{|\Im(\theta)| < \gamma_c}  \left|\frac{1}{K^{|m|}}\frac{1}{K - 1/K} \right|.
    \end{equation}
    Thus, we only need to increase $N$ as fast as $n$ to maintain the same accuracy.
    \label{rmk:m_effect_conv}
\end{remark}
\begin{corollary}
Let $\eta \in (0, c^2/\alpha_1)$, $N > n \geq 0$. Following the definitions of \cref{thm:spconv_1d}, we have
\begin{equation}
    |I_N - I| \leq \frac{4\pi M_\eta}{e^{\gamma_\eta (N-n)} - e^{-\gamma_\eta n}},
\end{equation}
where
    \begin{equation}
        \gamma_\eta = \log \left(1 + \frac{\eta}{2} + \sqrt{\left(1+\frac{\eta}{2}\right)^2 - 1} \right), \quad M_\eta = \frac{1}{2\sqrt{c^2/\alpha_1 - \eta}}.
    \end{equation}
    \label{cor:AprioriEst_cor}
\end{corollary}
\begin{proof}[Proof of \cref{cor:AprioriEst_cor}]
    Since $\eta \in (0, c^2/\alpha_1)$, we have 
    \begin{equation}
        \gamma_\eta = \log \left(1 + \frac{\eta}{2} + \sqrt{\left(1+\frac{\eta}{2}\right)^2 - 1} \right) < \log \left(1 + \frac{c^2}{2\alpha_1} + \sqrt{\left(1+\frac{c^2}{2\alpha_1}\right)^2 - 1} \right)
    \end{equation}
    By \cref{thm:sp_conv_lgf} and \cref{rmk:m_effect_conv}, we have
    \begin{equation}
        |I_N - I| \leq \frac{4\pi M_\eta'}{e^{\gamma_\eta (N - n)} - e^{-\gamma_\eta n}}
    \end{equation}
    where $M_\eta'$ is
    \begin{equation}
        M_\eta' = \sup\limits_{|\Im(\theta)| < \gamma_\eta} \left| \frac{1}{K^{|m|}}\frac{1}{K - 1/K} \right|.
    \end{equation}
    We can first put an upper bound on $|\cos(\theta)|$ within the strip $|\Im(\theta)| < \gamma_\eta$ as
    \begin{equation}
        |\cos(\theta)| \leq |e^{\gamma_\eta}/2| + |e^{-\gamma_\eta}/2| = 1 + \eta/2.
    \end{equation}
    By \cref{lm:nonReal}, $\phi^2 - 4 \notin \mathbb{R}_{< 0}$  within the strip $|\Im(\theta)| < \gamma_\eta$. Thus, for $\sqrt{\phi^2 - 4}$, we only take the principal branch. That is,
    \begin{equation}
        \sqrt{\phi^2 - 4} = |\phi^2 - 4|^{1/2}e^{i\Arg(\phi^2 - 4)/2}.
    \end{equation}
    Since $\Arg(\phi^2 - 4) \in (-\pi, \pi)$, we have
    \begin{equation}
        \Re(\sqrt{\phi^2 - 4}) > 0.
    \end{equation}
    Also,
    \begin{equation}
        \Re(\phi) \geq \lambda + c^2 - 2|\alpha_1\cos(\theta)| \geq \lambda + c^2 - 2\alpha_1(1+\eta/2) = 2 + (c^2 - \eta\alpha_1) > 2.
    \end{equation}
    Combined with the previous result, we obtain a lower bound of $|K|$
    \begin{equation}
       \Re(\phi + \sqrt{\phi^2 - 4}) > 2, \quad \textrm{and} \quad |K| = \left\lvert\frac{\phi + \sqrt{\phi^2 - 4}}{2}\right\lvert > 1.
    \end{equation}
    Then we can write
    \begin{equation}
        M_\eta' = \sup\limits_{|\Im(\theta)| < \gamma_\eta} \left| \frac{1}{K^{|m|}}\frac{1}{K - 1/K} \right| \leq \sup\limits_{|\Im(\theta)| < \gamma_\eta} \left| \frac{1}{K - 1/K} \right| = \sup\limits_{|\Im(\theta)| < \gamma_\alpha} \left| \frac{1}{\sqrt{\phi^2 - 4}} \right|.
    \end{equation}
    Further, we can put a lower bound on $|\phi^2 - 4|$, i.e.
    \begin{align*}
        |\phi^2 - 4| &= |\phi - 2||\phi + 2| \\
        &= |\lambda + c^2 -2\alpha_1\cos{\theta} - 2||\lambda + c^2 -2\alpha_1\cos{\theta} + 2| \\
        &\geq |\lambda - 2 + c^2 -2\alpha_1(1 + \eta/2)||\lambda + 2 + c^2 -2\alpha_1(1 + \eta/2)| \\
        &=(c^2 - \eta\alpha_1)(4 + c^2 - \eta\alpha_1) \\
        &> 4(c^2 - \eta\alpha_1).
    \end{align*}
    Thus, we can bound $M_\eta'$ using
    \begin{equation}
        M_\eta' = \sup\limits_{|\Im(\theta)| < \gamma_\eta} \left| \frac{1}{K^{|m|}}\frac{1}{K - 1/K} \right| \leq \frac{1}{2\sqrt{c^2 - \eta\alpha_1}} := M_\eta.
    \end{equation}
    Finally, we conclude with
    \begin{equation}
        |I_N - I| \leq \frac{4\pi M_\eta'}{e^{\gamma_\eta (N - n)} - e^{-\gamma_\eta n}} \leq \frac{4\pi M_\eta}{e^{\gamma_\eta (N-n)} - e^{-\gamma_\eta n}}.
    \end{equation}
\end{proof}
\begin{remark}
    Although the bound in \cref{cor:AprioriEst_cor} is looser than the one given in \cref{thm:sp_conv_lgf}, it provides us with an \textit{a priori} estimation of the error from the trapezoidal rule approximation of $B_c$, depending only on $N-n$ and $\eta$. This bound will be useful when we introduce the numerical framework to evaluate the trapezoidal rule approximation in the next section.
\end{remark}
\section{Fast Fourier Transform method for solving the lattice Green's function}
\label{sec:FFT_Alg}
As we have established the convergence rate of the trapezoidal rule approximation of $B_c$, we also notice that the specific form of the trapezoidal rule approximation of \cref{eq:integral_LGF_1D} is precisely the inverse discrete Fourier transform. As a result, one can utilize the inverse Fast Fourier Transform \cite{cooley1965algorithm} to efficiently evaluate the values of $B_c$. In this section, we introduce this algorithm. 
\subsection{A priori error estimate}
Fix $m \geq 0$. To evaluate $B_c(n,m)$ for arbitrary $n$, for a prescribed error tolerance $\epsilon$, we need to estimate the size of our trapezoidal rule approximation. Let $B_c^N(n,m)$ denote the $N$-term trapezoidal rule approximation of $B_c(n,m)$. By \cref{cor:AprioriEst_cor}, for any $\eta \in (0, c^2/\alpha_1)$, we can have an error estimate:
\begin{equation}
    |B_c(n,m) - B^N_c(n,m)| \leq \frac{2 M_\eta}{e^{\gamma_\eta (N-n)} - e^{-\gamma_\eta n}} \approx \frac{2 M_\eta}{e^{\gamma_\eta (N-n)}},
    \label{eq:adhoc_bound}
\end{equation}
where
\begin{equation}
    \gamma_\eta = \log \left(1 + \frac{\eta}{2} + \sqrt{\left(1+\frac{\eta}{2}\right)^2 - 1} \right), \quad M_\eta = \frac{1}{2\sqrt{c^2/\alpha_1 - \eta}}.
\end{equation}
Theoretically, one can solve for the optimal $\eta$ that minimizes the number of quadrature points in the trapezoidal rule approximation and obtain the corresponding least number of quadrature points needed, $N_{opt}(\epsilon)$. However, solving for $N_{opt}(\epsilon)$ involves solving a complicated nonlinear equation that might also induce numerical instability. 

Alternatively, one can approximate $N_{opt}(\epsilon)$ by fixing a small parameter $\delta > 0$ and let $\eta = (1-\delta)^2c^2/\alpha_1$. The corresponding minimum $N$ to satisfy a prescribed error tolerance $\epsilon$ is
\begin{equation}
    N_{ap}(\epsilon) = \left\lceil\frac{1}{\gamma_\eta}\log\left(\frac{1}{\epsilon (c/\sqrt{\alpha_1})\sqrt{2\delta - \delta^2}}\right) + n \right\rceil. 
    \label{eq:N_approx}
\end{equation}
In \cref{tab:approxNvsTrue}, we show the minimum number of quadrature points needed for the trapezoidal rule approximation to reach three different accuracy levels with different $c$. We computed these minimum quadrature points by finding the smallest $N$ that satisfies
\begin{equation}
    \frac{2 M_\eta}{e^{\gamma_\eta N}} \leq \epsilon.
\end{equation}
To solve for $N_{opt}$, we solve the following optimization problem:
\begin{equation}
\begin{aligned}
     \arg \min\limits_{N, \eta}& \quad N^2 \\
     \text{subject to}& \quad \log\left(\frac{2 M_\eta}{\epsilon e^{\gamma_\eta N}}\right) \leq 0 \\
                      & \quad \eta \in \left(0, \frac{c^2}{\alpha_1}\right).
\end{aligned}
\end{equation}
$N_{opt}$ is the ceiling of the $N$ from the solution of the above optimization. We computed $N_{opt}$ using MATLAB's \verb!fmincon! function  \cite{MATLAB:2020} with the optimization algorithm set to \verb!SQP! \cite{nocedal1999numerical} and the \verb!MaxFunctionEvaluations! argument set to $10^5$. We computed $N_{ap}$ using $\eta = (1-0.01)^2c^2/\alpha_1$ and applying \cref{eq:N_approx}. In the last rows of \cref{tab:approxNvsTrue}, we also report $N_{df}$, the $N_{opt}$ values obtained by using the default setting in MATLAB's \verb!fmincon! function. As shown in this table, in most cases, the differences between $N_{opt}$ and $N_{ap}$ are within $0.5\%$. Only one set of parameters yields a difference of $1.754\%$ at $c/\sqrt{\alpha_1} = 0.5$ and $\epsilon = 10^{-11}$. This large relative difference is the result of taking the ceiling when computing $N_{opt}$ and $N_{ap}$. Before taking the ceiling, their difference is approximately 0.25\%. Thus, the minimum $N$ to satisfy a prescribed error tolerance is well approximated using \cref{eq:N_approx}. 

\begin{table}[tbhp]
\footnotesize
\caption{Estimations of the minimum number of quadrature points to achieve an error tolerance of $10^{-14}$,$10^{-11}$, and $10^{-8}$. RE denotes the relative error, i.e. $(N_{ap} - N_{opt})/N_{opt}$. Superscript $^*$ denotes unconverged results.}\label{tab:approxNvsTrue}
\begin{center}
\scalebox{1}{
\begin{tabular}{|l|llll|llll|llll|}
\hline
$\epsilon$   & \multicolumn{4}{l|}{1E-14}  & \multicolumn{4}{l|}{1E-11}  & \multicolumn{4}{l|}{1E-8}  \\ \hline
$c/\sqrt{\alpha_1}$     & \multicolumn{1}{l|}{$N_{opt}$}     & \multicolumn{1}{l|}{$N_{ap}$}     & \multicolumn{1}{l|}{RE}        & $N_{df}$ & \multicolumn{1}{l|}{$N_{opt}$}     & \multicolumn{1}{l|}{$N_{ap}$}     & \multicolumn{1}{l|}{RE}        & $N_{df}$ & \multicolumn{1}{l|}{$N_{opt}$}     & \multicolumn{1}{l|}{$N_{ap}$}     & \multicolumn{1}{l|}{RE}        & $N_{df}$    \\ \hline
0.001 & \multicolumn{1}{l|}{41510} & \multicolumn{1}{l|}{41518} & \multicolumn{1}{l|}{0.02\%} & $-15^*$  & \multicolumn{1}{l|}{34511} & \multicolumn{1}{l|}{34541} & \multicolumn{1}{l|}{0.09\%} & $50844^*$ & \multicolumn{1}{l|}{27490} & \multicolumn{1}{l|}{27563} & \multicolumn{1}{l|}{0.27\%} & $531^*$  \\ \hline
0.005 & \multicolumn{1}{l|}{7977}  & \multicolumn{1}{l|}{7979}  & \multicolumn{1}{l|}{0.03\%} & $225^*$  & \multicolumn{1}{l|}{6576}  & \multicolumn{1}{l|}{6583}  & \multicolumn{1}{l|}{0.11\%} & 6576  & \multicolumn{1}{l|}{5171}  & \multicolumn{1}{l|}{5188}  & \multicolumn{1}{l|}{0.33\%} & $4483^*$ \\ \hline
0.01  & \multicolumn{1}{l|}{3918}  & \multicolumn{1}{l|}{3920}  & \multicolumn{1}{l|}{0.05\%} & 3918 & \multicolumn{1}{l|}{3218}  & \multicolumn{1}{l|}{3222}  & \multicolumn{1}{l|}{0.12\%} & 3218  & \multicolumn{1}{l|}{2515}  & \multicolumn{1}{l|}{2524}  & \multicolumn{1}{l|}{0.36\%} & 2515 \\ \hline
0.05  & \multicolumn{1}{l|}{752}   & \multicolumn{1}{l|}{752}   & \multicolumn{1}{l|}{0\%}        & 752  & \multicolumn{1}{l|}{611}   & \multicolumn{1}{l|}{612}   & \multicolumn{1}{l|}{0.16\%} & 611   & \multicolumn{1}{l|}{471}   & \multicolumn{1}{l|}{473}   & \multicolumn{1}{l|}{0.43\%} & 471  \\ \hline
0.1   & \multicolumn{1}{l|}{369}   & \multicolumn{1}{l|}{369}   & \multicolumn{1}{l|}{0\%}        & 369  & \multicolumn{1}{l|}{299}   & \multicolumn{1}{l|}{300}   & \multicolumn{1}{l|}{0.33\%} & 299   & \multicolumn{1}{l|}{229}   & \multicolumn{1}{l|}{230}   & \multicolumn{1}{l|}{0.44\%} & 229  \\ \hline
0.5   & \multicolumn{1}{l|}{72}    & \multicolumn{1}{l|}{72}    & \multicolumn{1}{l|}{0\%}        & 72   & \multicolumn{1}{l|}{57}    & \multicolumn{1}{l|}{58}    & \multicolumn{1}{l|}{1.75\%} & 57    & \multicolumn{1}{l|}{43}    & \multicolumn{1}{l|}{43}    & \multicolumn{1}{l|}{0\%}        & 43   \\ \hline
1     & \multicolumn{1}{l|}{36}    & \multicolumn{1}{l|}{36}    & \multicolumn{1}{l|}{0\%}        & 36   & \multicolumn{1}{l|}{29}    & \multicolumn{1}{l|}{29}    & \multicolumn{1}{l|}{0\%}        & 29    & \multicolumn{1}{l|}{22}    & \multicolumn{1}{l|}{22}    & \multicolumn{1}{l|}{0\%}        & 22   \\ \hline
\end{tabular}}
\end{center}
\end{table}

When computing $N_{opt}$ using MATLAB's \verb|fmincon| function with default settings, the algorithm can diverge when $c/\sqrt{\alpha_1}$ is small. In reality, we cannot guarantee for which $c/\sqrt{\alpha_1}$ and $\epsilon$ the optimization algorithm will converge. As a result, using the approximation in \cref{eq:N_approx} to estimate the number of quadrature points we need is more robust.

\subsection{Fast Fourier Transform based fast evaluation algorithm}
Now, we introduce the algorithm to compute $B_c$. Suppose we would like to compute all the values of $B_c(n,m)$ for a fixed $m$ and a range of $n \in [0, L]$ within some error tolerance $\epsilon$.
We first define:
\begin{equation}
        f(\theta) = \frac{1}{ K^{|m|}} \frac{1}{K - 1/K}, \quad \phi = \lambda + c^2 - 2\alpha_1\cos(\theta), \quad K = \frac{\phi + \sqrt{\phi^2 - 4}}{2}.
        \label{eq:eval_fft_coef}
\end{equation}
With these, the algorithm to compute that set of $B_c$ is shown as \cref{alg:FFT_LGF}. 
\begin{algorithm}
\caption{FFT-Based LGF Evaluation}
\begin{algorithmic}[1]

\Procedure{Trapzoidal Rule with FFT}{$c, \alpha_1, L, m, \epsilon$}
\State Compute 
$N_{pts}'(\epsilon) = \left\lceil\frac{1}{2}\left(\frac{1}{\gamma_\alpha}\log\left(\frac{1}{\epsilon (c/\sqrt{\alpha_1})\sqrt{2\delta - \delta^2}}\right) + L\right)\right\rceil$
\State $N_{pts} = \max({N_{pts}'(\epsilon), L})$
\State Declare $\boldsymbol{v} \in \mathbb{R}^{N_{pts}}$
\For{k = 0,1,..., $N_{pts}-1$}

$\boldsymbol{v}(k) = f(\pi k/N_{pts})$
\EndFor
\State $\boldsymbol{B}$ = \verb|irFFT|($\boldsymbol{v}$, $2N_{pts}$)
\For{k = 0,1,..., $N_{pts}-1$}

$\boldsymbol{B}(k) = \boldsymbol{B}(k) + (-1)^k f(\pi)/(2N_{pts})$
\EndFor
\State Return $\boldsymbol{B}$

\EndProcedure
\end{algorithmic}
\label{alg:FFT_LGF}
\end{algorithm}
Since $B_c(n,m)$ is real for all $n$ and $m$, we can utilize the inverse real Fast Fourier Transform (\verb|irFFT|). Note that to correctly compute $B_c$ using \cref{alg:FFT_LGF}, we need to set the correct number of output in the \verb|irFFT| function to $2N_{pts}$. Otherwise, the resulting FFT algorithm is different from the trapezoidal rule approximation, and the numerical results do not converge to $B_c$ exponentially. Also, the correction term in the second for-loop results from comparing the formula of \verb|irFFT| against the formula of the trapezoidal rule approximation. The difference, $r$, is
\begin{equation}
    r = \frac{1}{2\pi}\frac{2\pi}{2N_{pts}}\frac{e^{ik\pi}}{K^{|m|}}\frac{1}{K - 1/K} = \frac{1}{2N_{pts}}\frac{(-1)^k}{K^{|m|}}\frac{1}{K - 1/K} = \frac{(-1)^k}{2N_{pts}}f(\pi).
\end{equation}

\cref{alg:FFT_LGF} is best applicable when $N_{ap} \approx L$. In that case, the average operation count to evaluate an entry of $B_c$ is
\begin{equation}
    OC_{avg} \sim \frac{O(L\log(N_{ap}))}{L} \sim O(\log(N_{ap})).
\end{equation}

By using FFT, one can take advantage of the ubiquity of the highly optimized FFT libraries such as FFTW \cite{frigo1998fftw} and cuFFT \cite{cuFFT}. Thus, we not only speed up our computation in terms of reduced computational complexity but also benefit from the optimization in the software and hardware aspects. In the case that $N_{ap} \gg L$, depending on the computer architecture, it might be more efficient to directly evaluate the trapezoidal rule approximation term by term. In that case, the average computational complexity is $O(N_{ap})$.

\section{Numerical experiments}
\label{sec:num_results}
In this section, we assess the performance of the trapezoidal rule approximation with FFT (\cref{alg:FFT_LGF}) and the direct trapezoidal rule approximation (without FFT) by comparing them to two existing methods: evaluating Appell's double hypergeometric function representation in \cref{thm:appell_LGF} and numerically integrating \cref{eq:bessel_representation2} using Gauss-Kronrod quadrature \cite{de1978adaptive}. All the computations are done on an Apple Silicon M1 chip. The code used in this section is available online\footnote{The code for all the numerical experiments in this section can be found in \url{https://github.com/WeiHou1996/Fast-Screened-Poisson-LGF}.}. 

First of all, we demonstrate the error bound \cref{eq:adhoc_bound} of the trapezoidal rule approximation. In \cref{fig:Err_comp_appell}, we compare the trapezoidal rule approximation with the Appell's double hypergeometric function representation \cref{eq:appell_LGF} at $c = 0.3$ for a range of $N_{pts}$ (the number of quadrature points) and $\alpha_1$. With only 40 quadrature points in the trapezoidal rule approximation, the error converges to less than $10^{-7}$ across all the cases. The small errors establish the validity of the trapezoidal rule approximation and our implementation. When $c < 0.3$, we cannot evaluate Appell's double hypergeometric function within a reasonable amount of time. Thus, in that case, we use the trapezoidal rule approximation with a sufficiently high number of quadrature points as the reference value. In \cref{fig:Err_comp_Trap_001}, we demonstrate the error of the trapezoidal rule approximation at $c/\sqrt{\alpha_1} = 0.01$. In two of the subfigures, the absolute error violates the error bound when $n$ is small. However, those errors are below $10^{-13}$, indicating the effects of the finite precision arithmetic.

In the exercise above, we used Python's \verb|mpmath| \cite{mpmath} package to evaluate the Appell's double hypergeometric function, and we used \verb|NumPy| \cite{harris2020numpy} to evaluate the direct trapezoidal rule approximation. We do not directly compare the runtime of these two methods as the underlying numerical packages are implemented using different programming languages. However, as a point of reference, the time to evaluate $B_c(n,m)$ for all $(n,m) \in [0,9]^2$ at $c = 0.3$ using the Appell's double Hypergeometric function representation takes 7.22 seconds while evaluating the trapezoidal rule approximation (without FFT) to an absolute error below $10^{-10}$ takes 0.00216 seconds.
\begin{figure}
  \centering
  \subfloat[$N_{pts} = 20$, and $\alpha_1 = 1$.]{\includegraphics[width=0.5\textwidth]{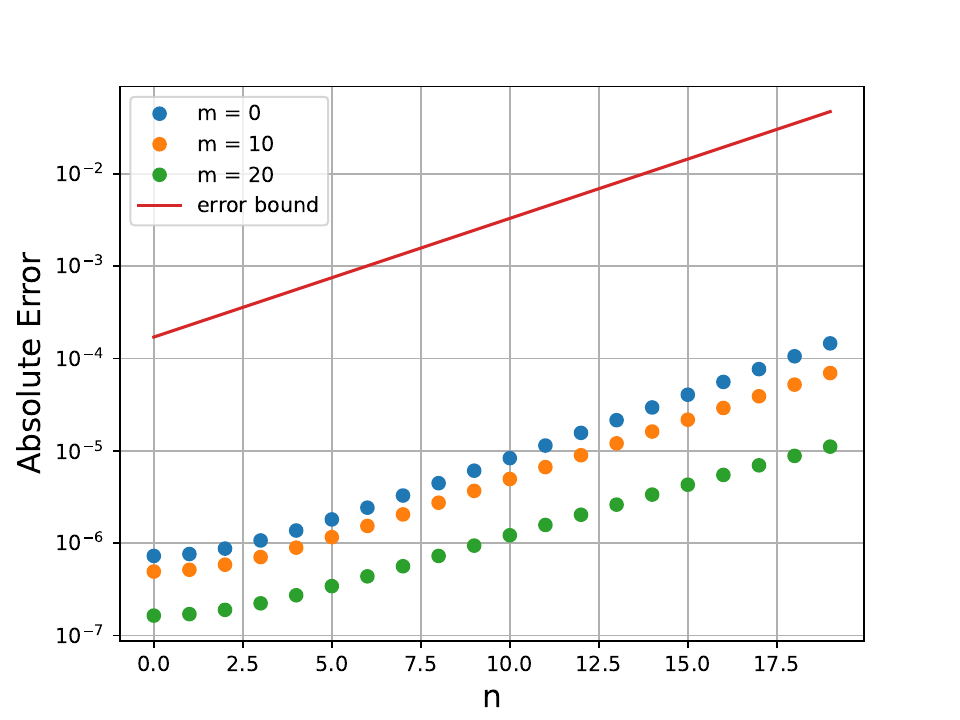}}
  \subfloat[$N_{pts} = 20$, and $\alpha_1 = 0.64$.]  {\includegraphics[width=0.5\textwidth]{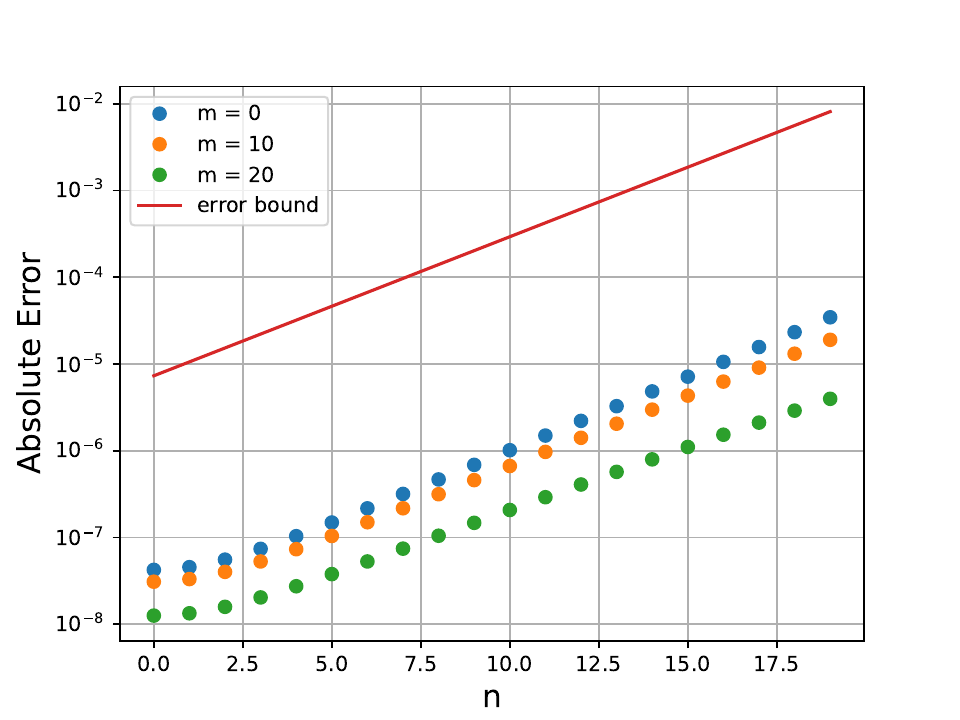}}
  \hfill
  \subfloat[$N_{pts} = 40$, and $\alpha_1 = 1$.]{\includegraphics[width=0.5\textwidth]{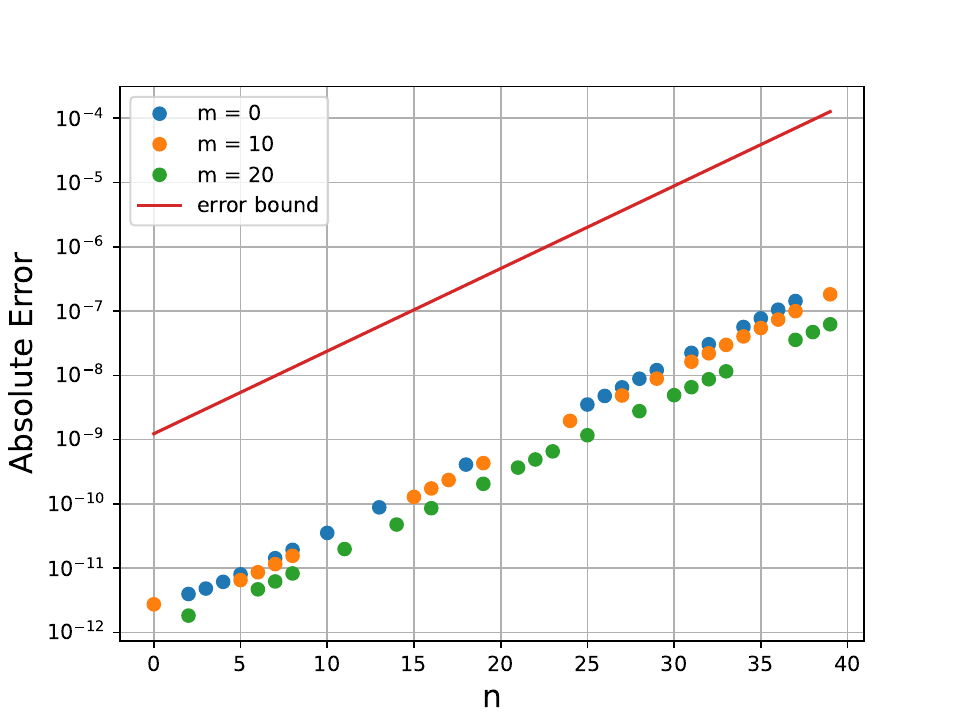}}
  \subfloat[$N_{pts} = 40$, and $\alpha_1 = 0.64$.]{\includegraphics[width=0.5\textwidth]{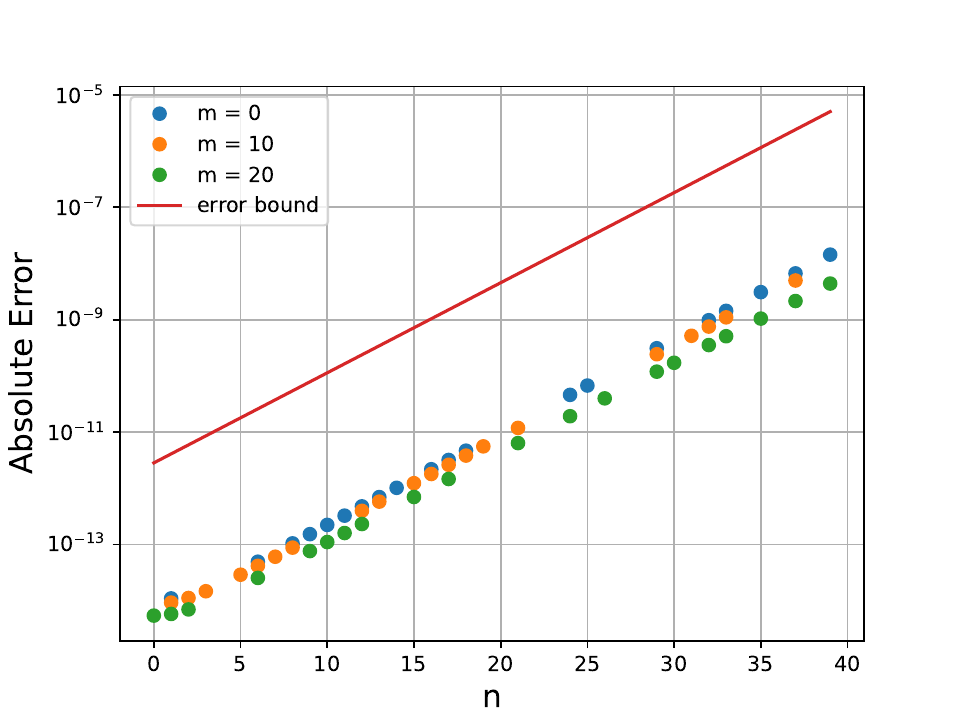}}
  \caption{Error of the trapezoidal rule approximation of $B_c(n,m)$ with various $N_{pts}$, $\alpha_1$, $n$, and $m$. Across all the cases, $c = 0.3$. The error is computed by referencing the analytical expression using \cref{eq:appell_LGF}. The error bound is computed using \cref{eq:adhoc_bound}.}
  \label{fig:Err_comp_appell}
\end{figure}
\begin{figure}
  \centering
  \subfloat[$N_{pts} = 1000$, $\alpha_1 = 1$.]{\includegraphics[width=0.5\textwidth]{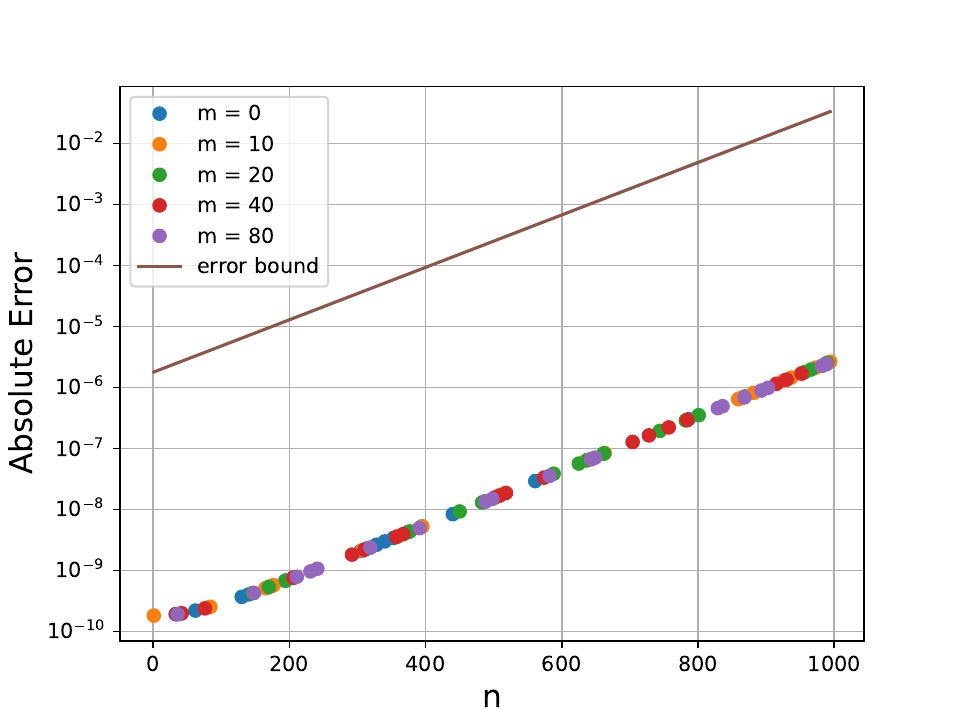}}
  \subfloat[$N_{pts} = 2000$, $\alpha_1 = 0.64$.]  {\includegraphics[width=0.5\textwidth]{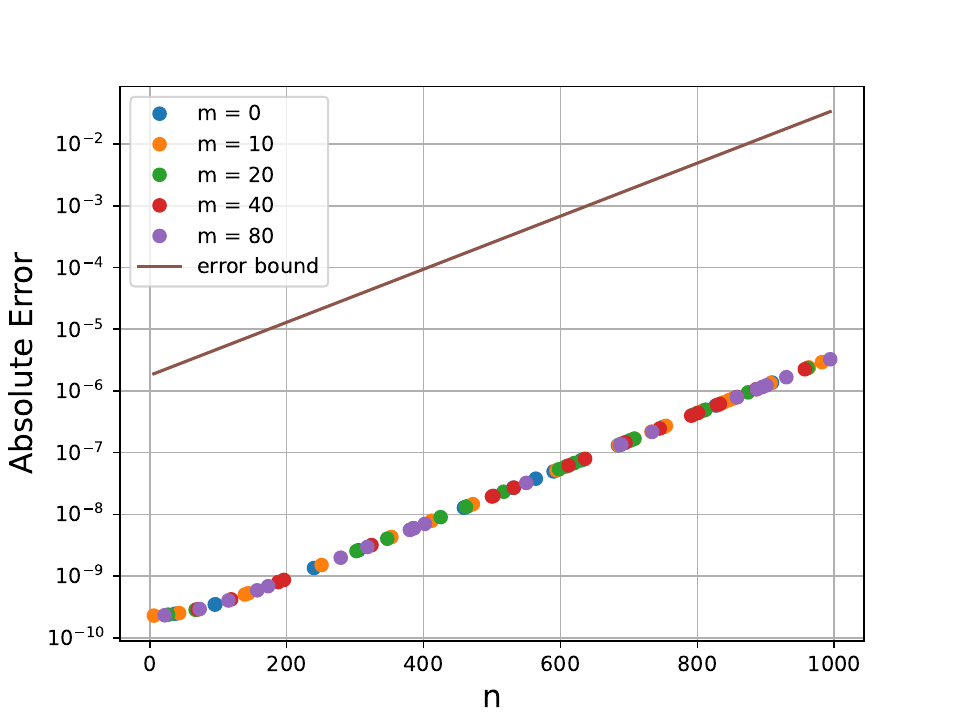}}
  \hfill
  \subfloat[$N_{pts} = 1000$, $\alpha_1 = 1$.]{\includegraphics[width=0.5\textwidth]{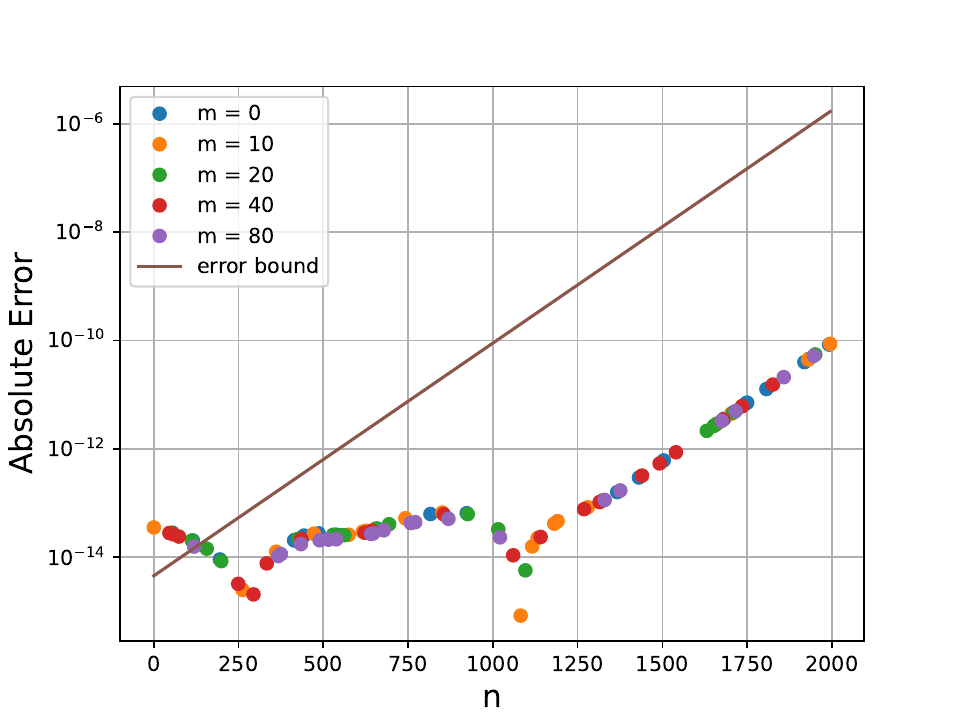}}
  \subfloat[$N_{pts} = 2000$, $\alpha_1 = 0.64$.]{\includegraphics[width=0.5\textwidth]{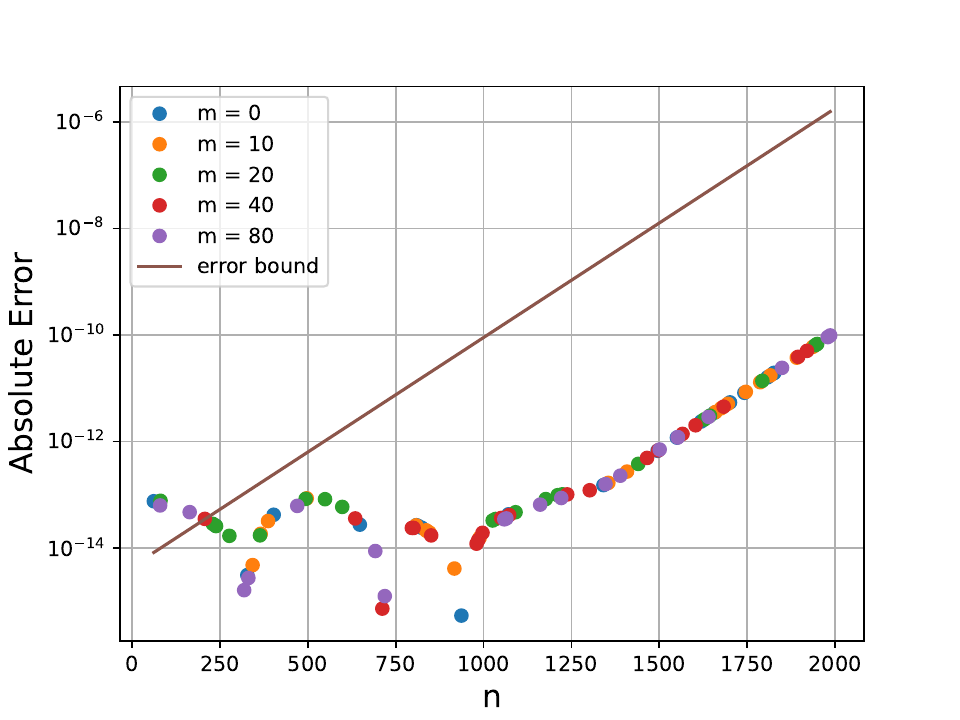}}
  \caption{Error of the trapezoidal rule approximation of $B_c(n,m)$ with various $N_{pts}$, $\alpha_1$, $n$, and $m$. Across all the cases, $c/\sqrt{\alpha_1} = 0.01$. The error is computed by referencing the trapezoidal rule approximation with $10,000$ quadrature points. }
  \label{fig:Err_comp_Trap_001}
\end{figure}

We also compare our algorithms with numerically integrating the Bessel function representation using \cref{eq:bessel_representation2}. Specifically, we used \verb|SciPy|'s \verb|scipy.integrate.quad| function and \verb|scipy.special.ive| function to numerically integrate \cref{eq:bessel_representation2}. Integrals over finite ranges are computed using the Gauss-Kronrod quadrature. Integrals with infinite ranges are first mapped onto a finite interval and then evaluated using the Gauss-Kronrod quadrature \cite{de1978adaptive}. We measure the performance of the trapezoidal rule approximation with and without FFT by defining a speedup factor compared to evaluating the LGF using the Bessel function representation. Given a set of values of LGF to compute, the speedup factor for a specific method, $\mathcal{M}$, is:
\begin{equation}
    \textrm{Speedup Factor} = \frac{\textrm{Runtime using the Bessel function representation}}{\textrm{Runtime using method }\mathcal{M}}.
\end{equation}
\begin{figure}
  \centering
  \subfloat[Average absolute error.]{\includegraphics[width=0.5\textwidth]{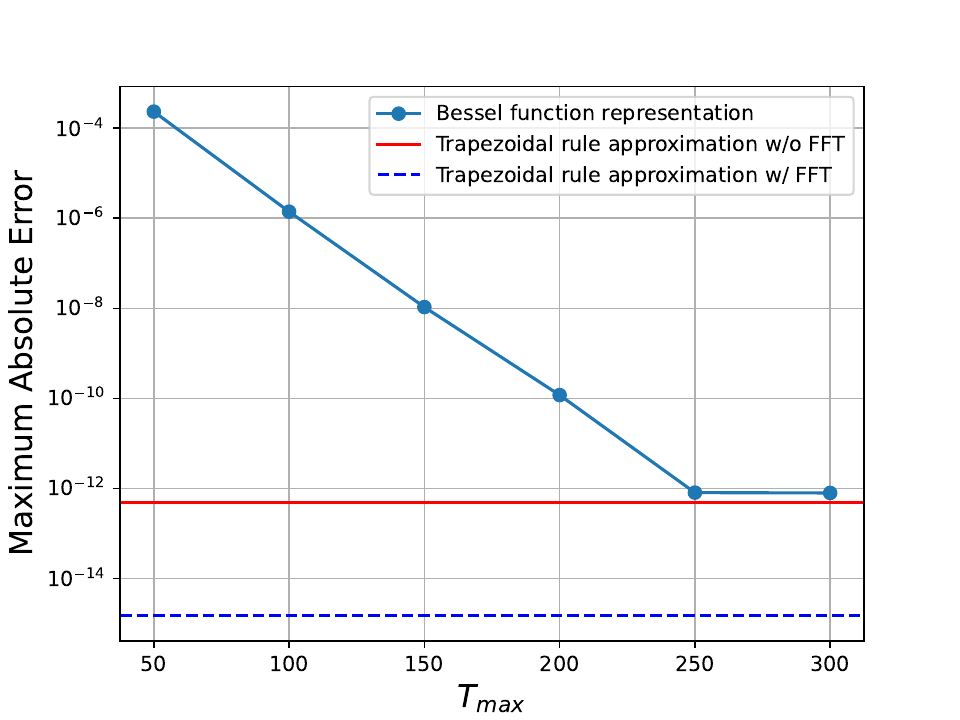}}
  \subfloat[Runtime speedup.]{\includegraphics[width=0.5\textwidth]{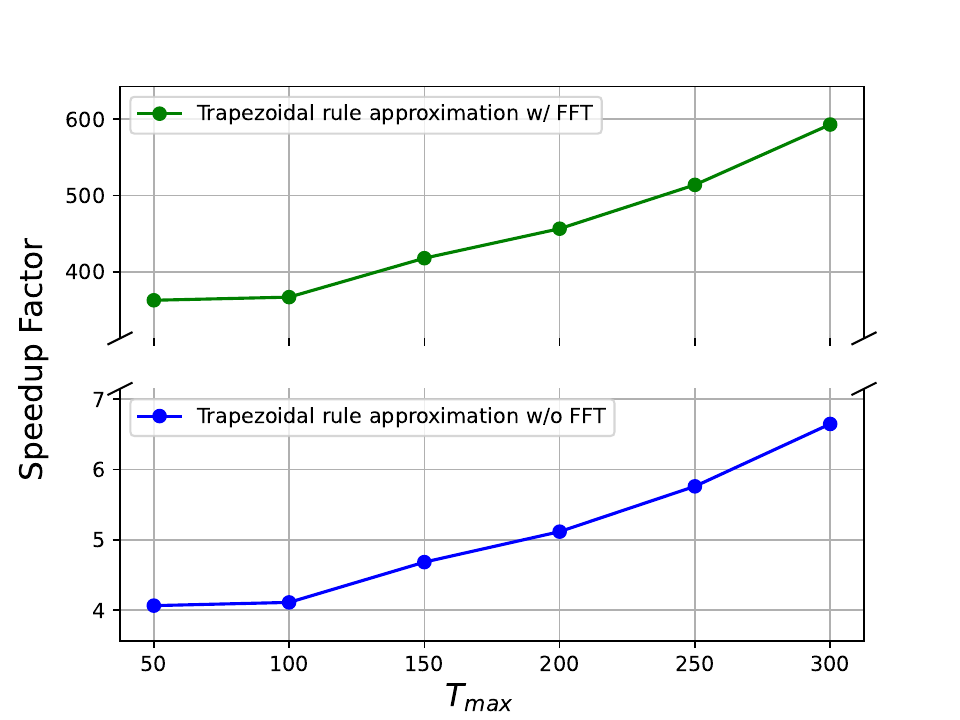}}
  \caption{The maximum absolute error and speedup factors of computing $B_c(n,m)$ when $c  = 0.3$ and $\alpha_1 = 0.5$. The error is computed by comparing the numerical integration results to the trapezoidal rule approximations with $2N_{ap}(\epsilon)$ quadrature points. }
  \label{fig:BesslComp_03}
\end{figure}
\begin{figure}
  \centering
  \subfloat[Average absolute error.]{\includegraphics[width=0.5\textwidth]{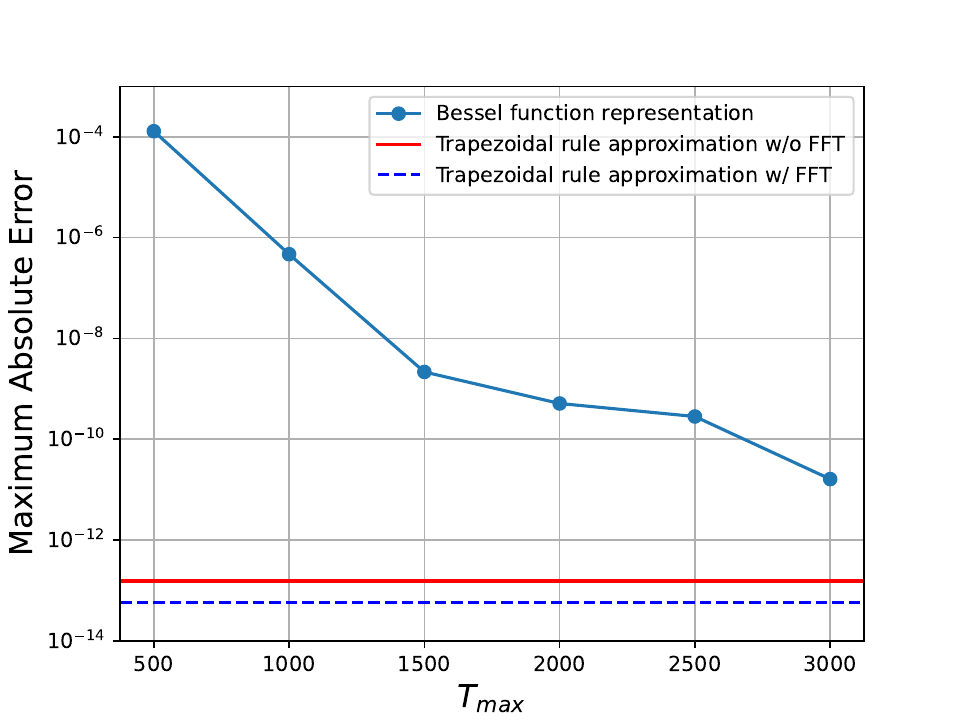}}
  \subfloat[Runtime speedup.]{\includegraphics[width=0.5\textwidth]{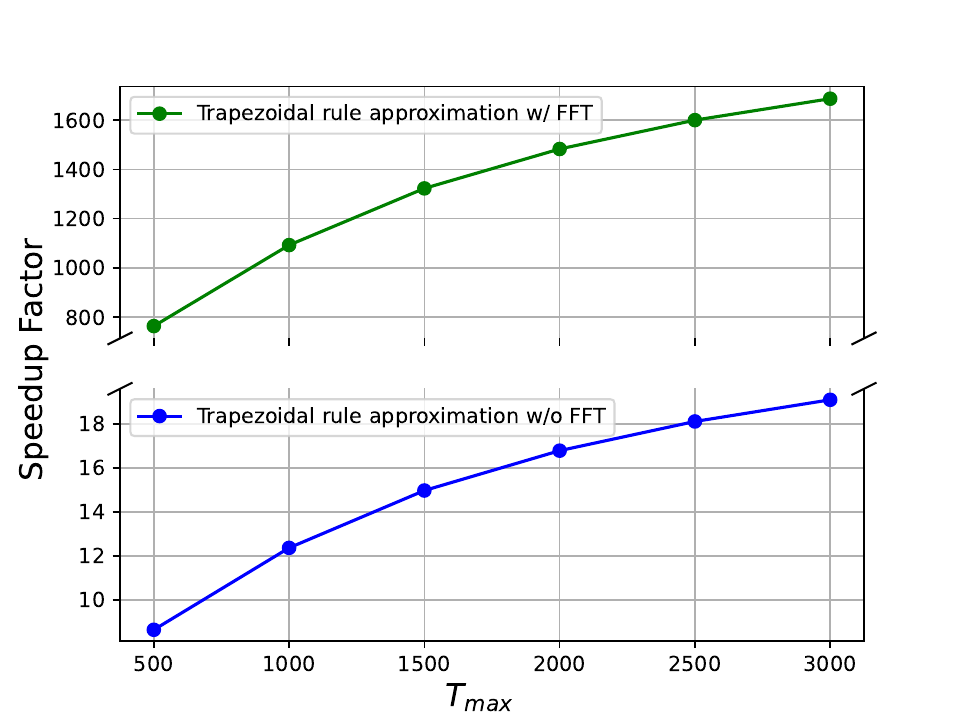}}
  \caption{The maximum absolute error and speedup factors of computing $B_c(n,m)$ when $c  = 0.1$ and $\alpha_1 = 0.5$. The error is computed by comparing the numerical integration results to the trapezoidal rule approximations with $2N_{ap}(\epsilon)$ quadrature points. }
  \label{fig:BesslComp_01}
\end{figure}
\begin{figure}
  \centering
  \subfloat[Maximum absolute error.]{\includegraphics[width=0.5\textwidth]{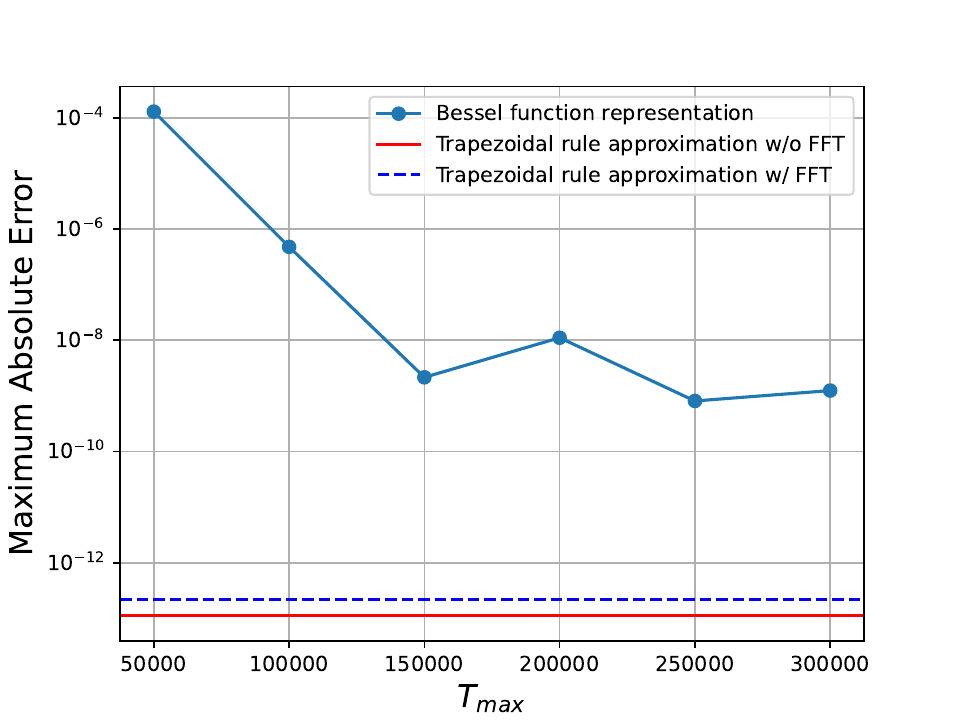}}
  \subfloat[Runtime speedup.]{\includegraphics[width=0.5\textwidth]{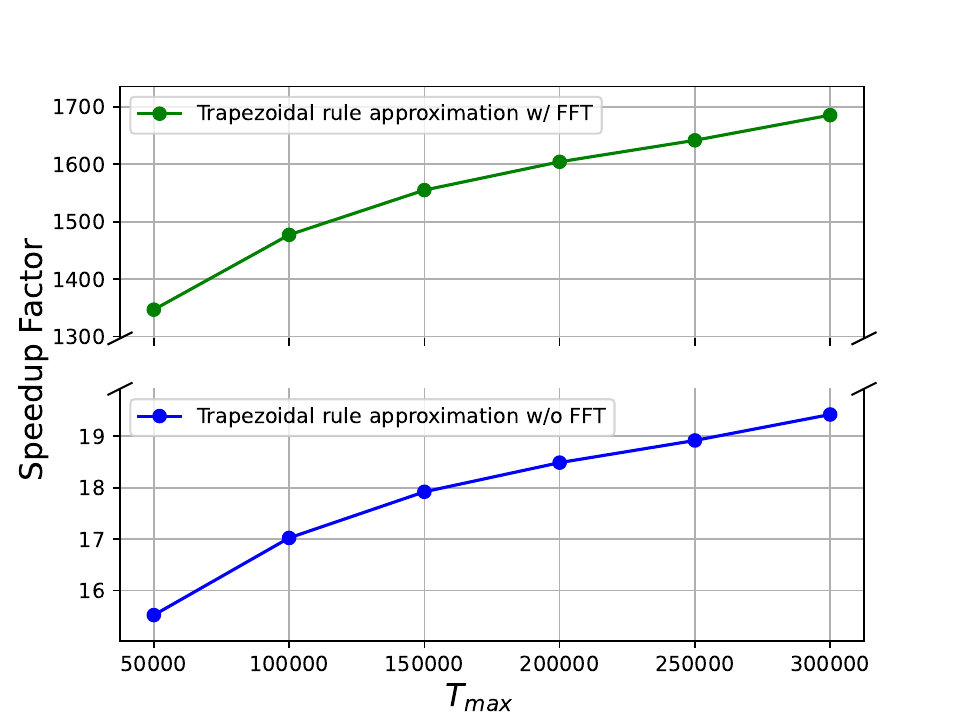}}
  \caption{The maximum absolute error and speedup factors of computing $B_c(n,m)$ when $c  = 0.01$ and $\alpha_1 = 0.5$. The error is computed by comparing the numerical integration results to the trapezoidal rule approximations with $2N_{ap}(\epsilon)$ quadrature points. }
  \label{fig:BesslComp_001}
\end{figure} 

In \cref{fig:BesslComp_03,fig:BesslComp_01,fig:BesslComp_001}, we compare our algorithms with numerically integrating the Bessel function representation with a large integration upper bound, presenting both error and speedup.  We set $c = 0.3, 0.1, 0.01$ and $\alpha_1 = 0.5$ and evaluate all $B_c(n,m)$ for $(n,m) \in [0, 99]^2$ using three methods: numerically integrating the Bessel function representation with large and varying integration upper bounds ($T_{max}$), evaluating the trapezoidal rule approximation directly without FFT, and evaluating the values of the LGF in batch using FFT (\cref{alg:FFT_LGF}). In all these methods, the absolute error tolerance is set to $10^{-10}$. When $c = 0.3$, the speedup factor is around 6 for the trapezoidal rule approximation without FFT and around 500 for the trapezoidal rule approximation with FFT. When $c = 0.1$, the advantage of the trapezoidal rule approximation is more prominent. The speedup factors are around 20 and 1650 without and with FFT, respectively. At $c = 0.01$, numerically integrating the Bessel function representation cannot reach satisfactory accuracy. In contrast, the trapezoidal rule approximations are able to reach the desired accuracy, with significant speedup factors. These three cases demonstrate the efficiency and robustness of the trapezoidal rule approximation and \cref{alg:FFT_LGF}.

Finally, in \cref{fig:BesslInf}, we present the error and speedup of our algorithm compared to numerically evaluating the Bessel function by mapping the infinite integration interval to a finite one \cite{de1978adaptive}. We fix $\alpha_1 = 0.5$ and vary $c$ between $0.001$ and $0.2$. We evaluate the values of $B_c$ within the square $[0, 99]^2$ with an absolute error tolerance of $10^{-10}$. When evaluating $B_c$ using the Bessel function representation, the numerical quadrature can diverge when $c$ is small. However, this does not happen with the trapezoidal rule approximation. Also, even when converged, numerically integrating the Bessel function representation does not always satisfy the prescribed error tolerance. In contrast, the trapezoidal rule approximations not only consistently satisfy the error tolerance but also greatly reduce the runtime. The trapezoidal rule approximation reaches a speedup factor of at least 15 without using FFT and a speedup factor of at least 1000 when using FFT.

\begin{figure}
  \centering
  \subfloat[Maximum absolute error.]{\includegraphics[width=0.5\textwidth]{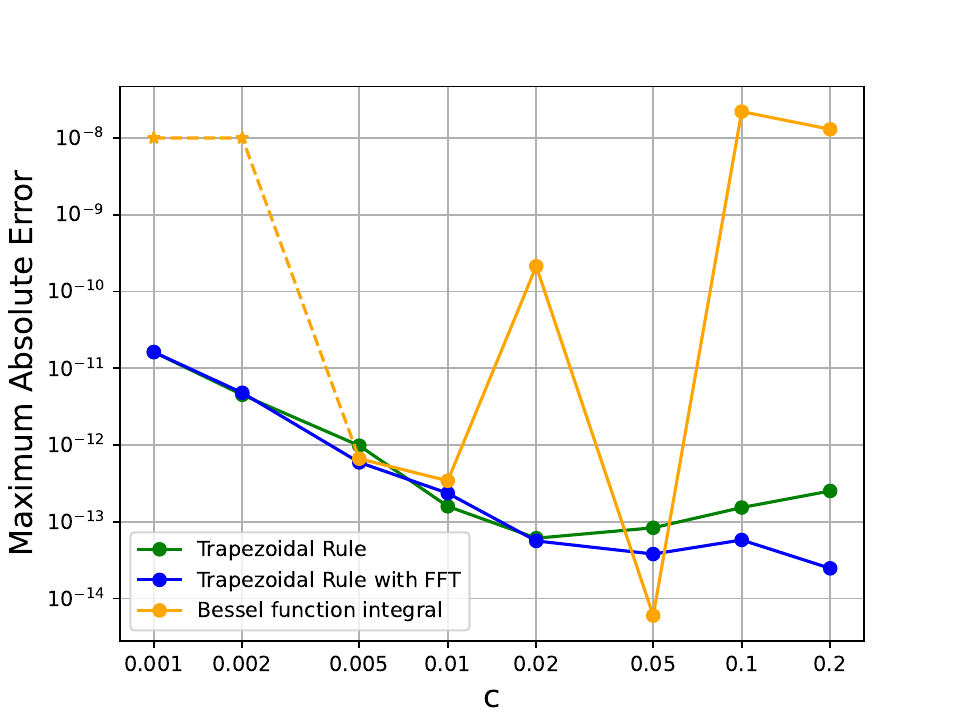}}
  \subfloat[Runtime speedup.]{\includegraphics[width=0.5\textwidth]{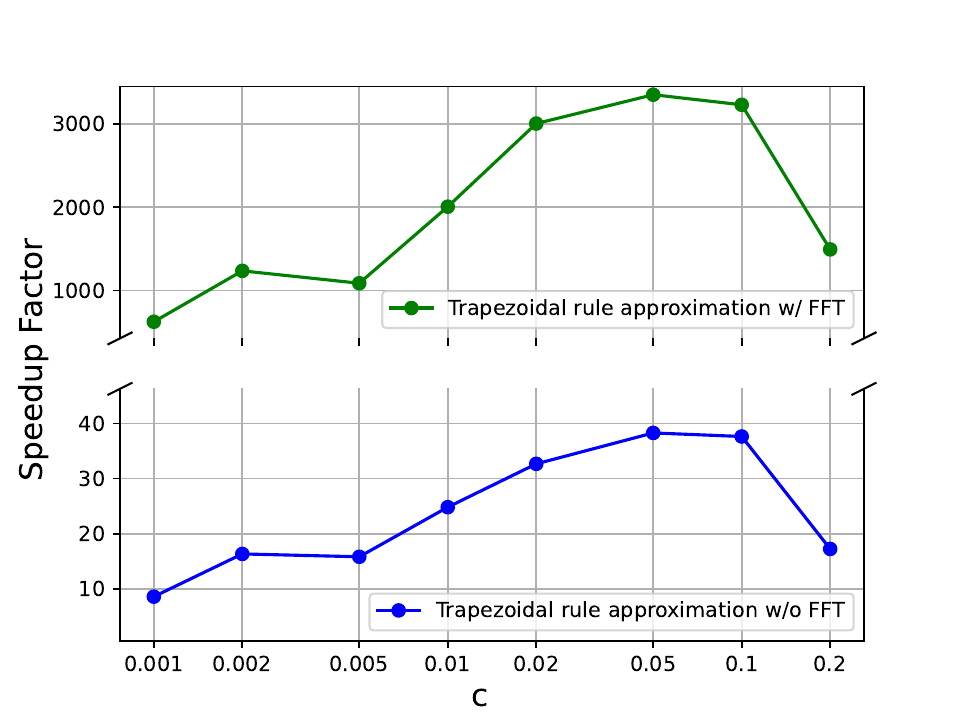}}
  \caption{The maximum error and speedup factors when computing $B_c(n,m)$ using the trapezoidal rule with FFT, without FFT, and numerically integrating the Bessel function representation using the transform proposed in \cite{de1978adaptive}. The maximum absolute error is obtained by comparing the values from evaluating the trapezoidal rule approximation of $B_c$ using $2N_{ap}(\epsilon)$ quadrature points. The stars and the dashed line in the maximum absolute error plot indicate that some values in the numerical integration did not converge. }
  \label{fig:BesslInf}
\end{figure}

\section{Application 1: lattice Green's function of the three-dimensional Poisson equation with one periodic direction}
\label{sec:app1}
An immediate application of $B_c$ is solving the discretized three-dimensional Poisson equation with one periodic direction. Consider a discretization of the three-dimensional Poisson equation with constant spatial resolutions $\Delta x_1$, $\Delta x_2$, and $\Delta x_3$ in each of the three spatial dimensions. The resulting discretized Poisson equation reads
\begin{equation}
    \sum\limits_{i = 1}^3 \left(\frac{2}{\Delta x_i ^ 2}u (\boldsymbol{n}) - \frac{1}{\Delta x_i ^ 2} u (\boldsymbol{n} + \boldsymbol{e}_i) - \frac{1}{\Delta x_i ^ 2} u (\boldsymbol{n} - \boldsymbol{e}_i) \right) =  f(\boldsymbol{n}).
\end{equation}
In addition, we assume that the solution is periodic in the third direction and unbounded in the first two directions. To solve the Poisson equation, we can find a corresponding LGF and apply discrete convolution. The LGF satisfies:
\begin{equation}
    \sum\limits_{i = 1}^3 \left(\frac{2}{\Delta x_i ^ 2} G (\boldsymbol{n}) - \frac{1}{\Delta x_i ^ 2} G (\boldsymbol{n} + \boldsymbol{e}_i) - \frac{1}{\Delta x_i ^ 2} G (\boldsymbol{n} - \boldsymbol{e}_i) \right) =  \delta^{\mathbb{Z}^3}(\boldsymbol{n})
\end{equation}
where $\delta^{\mathbb{Z}^d} : {\mathbb{Z}^d} \rightarrow \mathbb{R}$ is defined as:
\begin{equation}
    \delta^{\mathbb{Z}^d}(\boldsymbol{n}) \begin{cases}
    1 \quad \textrm{if } \boldsymbol{n} = \boldsymbol{0} \\
    0 \quad \textrm{otherwise}.
    \end{cases}
\end{equation}
This equation can be readily solved if we can solve the following equation
\begin{equation}
    \sum\limits_{i = 1}^3 \left( 2\alpha_i G (\boldsymbol{n}) - \alpha_i G (\boldsymbol{n} + \boldsymbol{e}_i) - \alpha_i G (\boldsymbol{n} - \boldsymbol{e}_i) \right) =  \delta^{\mathbb{Z}^3}(\boldsymbol{n})
    \label{eq:period_lgf}
\end{equation}
where $\alpha_2 = 1$, $\alpha_1 = \Delta x_2^2/ \Delta x_1^2$, and $\alpha_3 = \Delta x_2^2/ \Delta x_3^2$. 

Suppose the solution is assumed to be $N_p$ periodic in $n_3$, with $\boldsymbol{n} = [n_1, n_2, n_3]$, we can write
\begin{equation}
    G(\boldsymbol{n}) = G(n_1, n_2, n_3) = G(n_1, n_2, n_3 + N_p).
\end{equation}
We define the discrete Fourier transform of a $N_p$ periodic discrete function $f$ 
\begin{equation}
    \tilde{f}_k = \mathcal{F}[f](k) = \sum\limits_{n = 0}^{N_p - 1} f(n) e^{-i 2\pi kn/N_p}
\end{equation}
and its inverse
\begin{equation}
    f(n) = \mathcal{F}^{-1}[\tilde{f}](n) = \frac{1}{N_p}\sum\limits_{k = 0}^{N_p - 1} \tilde{f}_k e^{i 2\pi kn/N_p}.
\end{equation}
Thus, there exists a set of Fourier coefficients $\{\tilde{G}_k\}_k$ such that
\begin{equation}
    G(n_1, n_2, n_3) = \frac{1}{N_p}\sum\limits_{k = 0}^{N_p - 1} \tilde{G}_k(n_1, n_2) e^{i 2\pi kn_3/N_p}.
\end{equation}

And the LHS of \cref{eq:period_lgf} can be written as:
\begin{equation}
    \frac{1}{N_p}\left[\sum\limits_{k = 0}^{N_p - 1} e^{i 2\pi kn_3/N_p} L_{\kappa(k)}\tilde{G}_k(n_1, n_2)\right]
\end{equation}
where
\begin{equation}
    \kappa(k) = \sqrt{2\alpha_3 - 2\alpha_3 \cos(2\pi k /N_p)}.
\end{equation}
Applying discrete Fourier transform to both sides of \cref{eq:period_lgf} gives
\begin{equation}
    L_{\kappa(k)}\tilde{G}_k(n_1, n_2) = \delta_{0n_1}\delta_{0n_2}.
\end{equation}
By definition, $\tilde{G}_k(n_1, n_2) = B_{\kappa(k)}(n_1, n_2)$. Thus, we find
\begin{equation}
    G(n_1, n_2, n_3) = \frac{1}{N_p}\sum\limits_{k = 0}^{N_p - 1}  e^{i 2\pi kn_3/N_p} B_{\kappa(k)}(n_1, n_2).
\end{equation}
Using \cref{alg:FFT_LGF} and the approximation in \cref{eq:G_N_approx}, one can efficiently compute the values of $B_{\kappa(k)}$ and evaluate $G$ using inverse Fast Fourier Transform. We use this result to solve a Poisson equation $\nabla^2 \phi = -f$ with the following analytical solution
\begin{equation}
    \phi(x,y,z) = \frac{\exp(-64x^2 - 4y^2)}{2 - \cos(z)}.
\end{equation}
We obtain the source term $f$ by taking $-\nabla^2 \phi$. The computational domain is $[-1,1]\times [-4,4] \times [0,2\pi]$. The convergence result is shown in \cref{fig:3D_LGF_Conv}.
\begin{figure}
    \centering
    \includegraphics[width=0.5\textwidth]{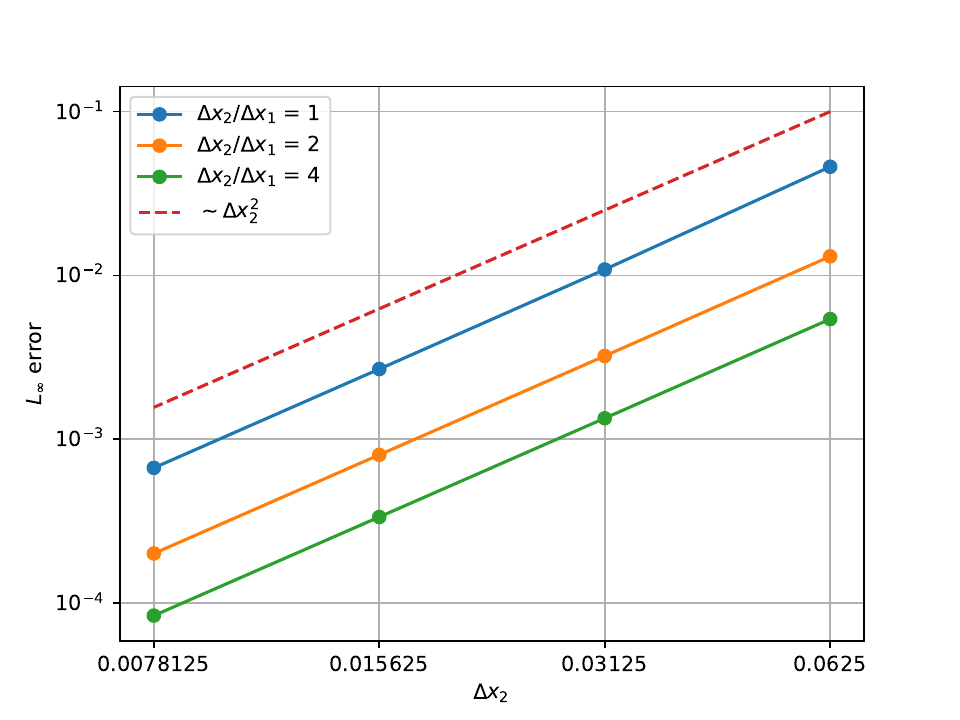}
    \caption{Convergence study of solving the Poisson equation using the three-dimensional Poisson LGF with one periodic direction. The ratio $\Delta x_3/\Delta x_2 = 2\pi$ is held constant across all cases. Within each series, the ratio between $\Delta x_1$ and $\Delta x_2$ is fixed. Different series have different ratios of $\Delta x_2$ and $\Delta x_1$. The dashed line indicates the expected second-order convergence rate.}
    \label{fig:3D_LGF_Conv}
\end{figure}

\section{Application 2: random walks with killing on a two-dimensional rectangular lattice}
\label{sec:app2}
    Consider a random walk with killing on a two-dimensional rectangular lattice \cite{madras1989random,lawler2010random}. When the walker is at location $(n,m)$, it can behave in five different ways with certain probabilities. It can either move up, down, left, or right for one step. It can also decide to stay at $(n,m)$ forever (i.e. killed). We assume that the probabilities are all strictly positive and defined by:
\begin{equation}
    \begin{aligned}
    &\mathbb{P}((n + 1,m)) = \mathbb{P}((n - 1,m)) = p_1, \\
    &\mathbb{P}((n,m + 1)) = \mathbb{P}((n,m - 1)) = p_2, \\
    &\mathbb{P}(\textrm{stay at $(n,m)$ forever}) =p_k := 1 - 2p_1 - 2p_2.
    \end{aligned}
\end{equation}
We can compute the probability of a random walk starting at an arbitrary point $(n,m)$ and eventually returning to the origin. Let this probability be denoted as $\rho(n,m)$. We can write 
\begin{equation}
    \rho(n,m) = p_1 \rho(n+1, m) + p_1 \rho(n-1, m) + p_2 \rho(n, m - 1) + p_2 \rho(n, m + 1)
    \label{eq:RandomWalkPoisson}
\end{equation}
with the terminal condition $\rho(0,0) = 1$. The above equation is satisfied everywhere in $\mathbb{Z}^2$ except at the origin. At the origin, we have
\begin{equation}
    \rho(0,0) = p_1 \rho(1, 0) + p_1 \rho(-1, 0) + p_2 \rho(0, - 1) + p_2 \rho(0, 1) + C(p_1, p_2)
    \label{eq:term_cond}
\end{equation}
where $C(p_1, p_2)$ is an undetermined function to satisfy the condition $\rho(0,0) = 1$. 

With $\alpha_1 = p_1/p_2$, we can rewrite the governing equation of $\rho(n,m)$ as
\begin{equation}
    L_{\kappa(p_1, p_2)} \rho(n,m) = \frac{1}{p_2}\delta_{0n}\delta_{0m}C(p_1, p_2)
\end{equation}
where
\begin{equation}
    \kappa(p_1, p_2) = \sqrt{\frac{1 - 2p_1 -2p_2}{p_2}}.
\end{equation}

By definition, we have 
\begin{equation}
    \rho(n,m) = \frac{1}{p_2}C(p_1, p_2) B_{\kappa(p_1, p_2)}(n,m).
\end{equation}
To determine $C(p_1, p_2)$, we use the terminal condition of $\rho(0,0) = 1$ and \cref{eq:term_cond} to find
\begin{equation}
    C(p_1, p_2) = \frac{1}{1 + \frac{2p_1}{p_2} B_{\kappa(p_1, p_2)}(1,0) + 2 B_{\kappa(p_1, p_2)}(0,1)}.
\end{equation}
In the equation above, we can compute $B_{\kappa(p_1, p_2)}(1,0)$ and $B_{\kappa(p_1, p_2)}(0,1)$ using the integral in \cref{thm:LGF_int_1D} through the trapezoidal rule approximation in \cref{eq:trap_appr}. Then we can compute the return probability at all other locations using either direct trapezoidal rule approximation or \cref{alg:FFT_LGF}. A sample return probability (with $p_1 = 0.2(1-p_k), p_2 = 0.3(1-p_k)$) with various $p_k$ is shown in \cref{fig:RW_kill}.
\begin{figure}
  \centering
  \subfloat[Return probability at various $m$ with $n = 0$.]{\includegraphics[width=0.5\textwidth]{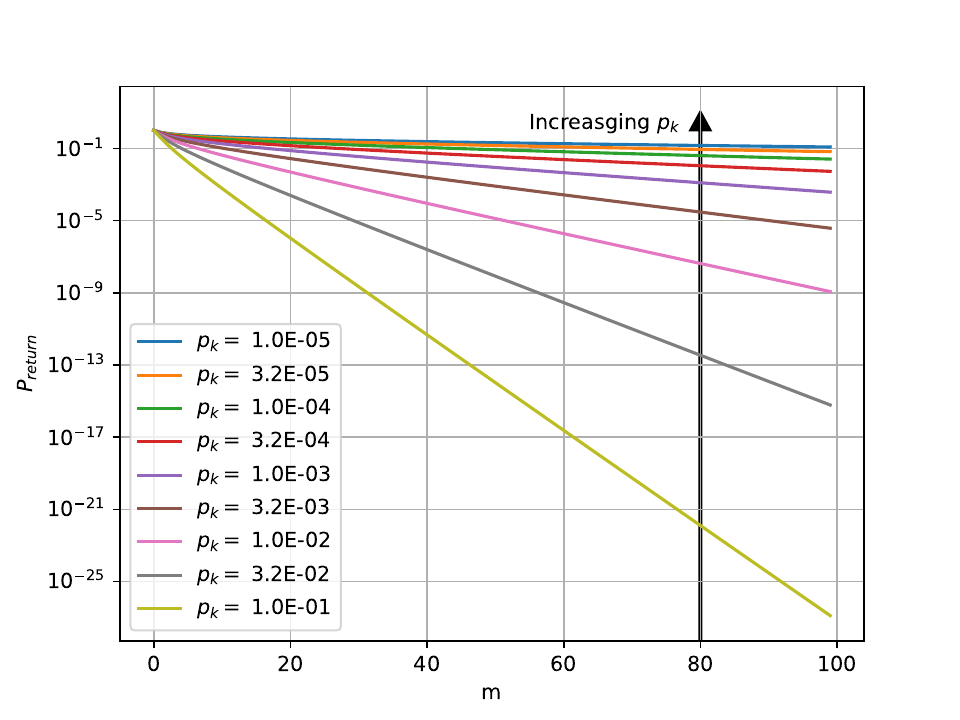}}
  \subfloat[Return probability on the diagonal ($n = m$).]{\includegraphics[width=0.5\textwidth]{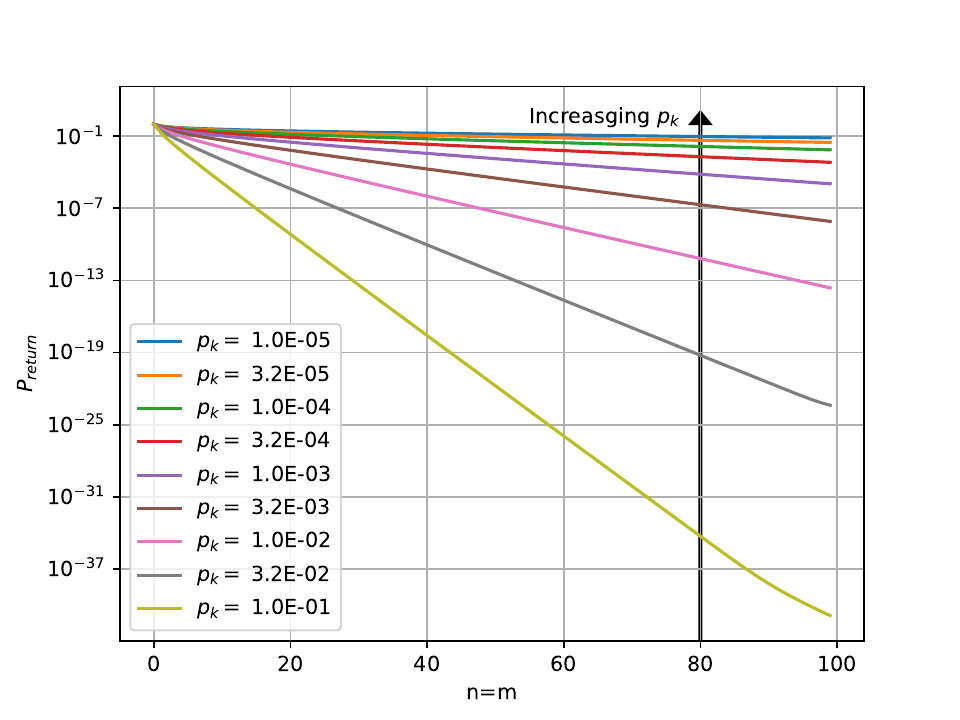}}
  \caption{The return probability, $P_{return}$, at various $n,m$ at different kill probabilities $p_k$.}
  \label{fig:RW_kill}
\end{figure}

\section{Conclusion}
In this paper, we studied the two-dimensional lattice Green's function (LGF) of the screened Poisson equation on rectangular lattices. In particular, we proposed two efficient ways to compute the LGF, depending on the $c^2$ term.

When $c^2$ is large, we conducted an asymptotic expansion to give an approximation formula of the LGF. We showed that this approximation exponentially converges towards the true values of the LGF. Using the approximation formula, we also established the decay rate of the LGF towards infinity.

Although the asymptotic expansion exponentially converges toward the entries of the LGF, when $c^2$ is small, approximating LGF using the asymptotic expansion becomes prohibitively expensive. To remedy this, we derived a one-dimensional integral representation of the LGF. In addition, we showed that the trapezoidal rule approximates this one-dimensional integral exponentially fast. By exploiting the properties of the integrand and the trapezoidal rule approximation, we devised a fast algorithm for batch-evaluating the values of the LGF using the Fast Fourier Transform. To enhance the algorithm's robustness, we proposed a simple yet accurate estimate of the minimum number of quadrature points needed for a prescribed error tolerance. Compared to existing formulations such as the Appell's double hypergeometric function representation and the Bessel function representation, the resulting algorithm demonstrates high robustness and efficiency when evaluating the LGF. 

Finally, we demonstrated how our algorithms can be efficiently used to tabulate the LGF and solve two application problems -- the three-dimensional Poisson equation with two unbounded directions and one periodic direction, and the return probability of a random walk with killing on a rectangular lattice.

\section{Acknowledgments} The authors would like to thank Prof. John Sader for the inspiring discussions.

\bibliographystyle{siamplain}
\bibliography{references}

\end{document}